\documentclass[3p, authoryear]{elsarticle}

\usepackage{amsmath}
\usepackage{amsfonts}
\usepackage{paralist}
\usepackage{booktabs}

\usepackage{paralist}
\usepackage{url}
\usepackage{color}
\usepackage{framed}
\usepackage{graphicx}
\usepackage{mathrsfs}
\usepackage[T1]{fontenc}
\usepackage[colorlinks,
            linkcolor=red,
            anchorcolor=blue,
            citecolor=green
            ]{hyperref}

\newcommand{\RR}{{\mathbb R}}

\newcommand{\re}{\mathbb{R}}

\newcommand{\N}{\mathbb{N}}

\renewcommand{\P}{\mathbf{P}}
\newcommand{\qm}{\mathbf{qmodule}}
\newcommand{\K}{\mathbf{K}}

\newcommand{\mL}{\mathscr{L}}

\newcommand{\mH}{\mathscr{H}}

\def\af{\alpha}

\newcommand{\ud}{\mathrm{d}}

\newcommand{\td}[1]{\tilde{#1}}

\newcommand{\rank}{\mbox{\upshape rank}}
\newcommand{\sos}{\mbox{\upshape\tiny sos}}
\newcommand{\psdp}{\mbox{\upshape\tiny psdp}}
\newcommand{\dsdp}{\mbox{\upshape\tiny dsdp}}

\newcommand{\bdes}{\begin{description}}
\newcommand{\edes}{\end{description}}
\newcommand{\bal}{\begin{align}}
\newcommand{\eal}{\end{align}}
\newcommand{\bnum}{\begin{enumerate}}
\newcommand{\enum}{\end{enumerate}}
\newcommand{\bit}{\begin{itemize}}
\newcommand{\eit}{\end{itemize}}
\newcommand{\bea}{\begin{eqnarray}}
\newcommand{\eea}{\end{eqnarray}}
\newcommand{\be}{\begin{equation}}
\newcommand{\ee}{\end{equation}}
\newcommand{\baray}{\begin{array}}
\newcommand{\earay}{\end{array}}
\newcommand{\bsry}{\begin{subarray}}
\newcommand{\esry}{\end{subarray}}
\newcommand{\bca}{\begin{cases}}
\newcommand{\eca}{\end{cases}}
\newcommand{\bcen}{\begin{center}}
\newcommand{\ecen}{\end{center}}
\newcommand{\bbm}{\begin{bmatrix}}
\newcommand{\ebm}{\end{bmatrix}}
\newcommand{\bmx}{\begin{matrix}}
\newcommand{\emx}{\end{matrix}}
\newcommand{\bpm}{\begin{pmatrix}}
\newcommand{\epm}{\end{pmatrix}}
\newcommand{\btab}{\begin{tabular}}
\newcommand{\etab}{\end{tabular}}

\newcommand{\wt}[1]{\widetilde{#1}}
\newcommand{\So}{\wt{S}_>}

\newcommand{\St}{\wt{S}}

\newcommand{\h}{\text{\upshape hom}}

\makeatletter
\def\munderbar#1{\underline{\sbox\tw@{$#1$}\dp\tw@\z@\box\tw@}}
\makeatother

\newtheorem{theorem}{Theorem}[section]

\newtheorem{prop}[theorem]{Proposition}
\newtheorem{condition}[theorem]{Condition}

\newtheorem{lemma}[theorem]{Lemma}
\newtheorem{cor}[theorem]{Corollary}

\newtheorem{assumption}[theorem]{Assumption}
\newtheorem{defi}[theorem]{Definition}

\newtheorem{definition}[theorem]{Definition}
\newtheorem{example}[theorem]{Example}

\newtheorem{remark}[theorem]{Remark}

\newenvironment{proof}{\text{Proof. }}{\hfill $\square$}
\makeatletter
\def\ps@pprintTitle{%
 \let\@oddhead\@empty
 \let\@evenhead\@empty
 \def\@oddfoot{\centerline{\thepage}}%
 \let\@evenfoot\@oddfoot}
\makeatother

\begin{document}
\begin{frontmatter}
\title{On semi-infinite systems of convex polynomial inequalities and
polynomial optimization problems}

\author[1]{Feng Guo}
\ead{fguo@dlut.edu.cn}

\address[1]{
School of Mathematical Sciences,\\
Dalian  University of Technology,
Dalian, 116024, China}

\author[2]{Xiaoxia Sun\corref{cor1}}
\address[2]{School of Mathematics,\\
Dongbei University of Finance and Economics,
Dalian, 116025, China}
\ead{xiaoxiasun@dufe.edu.cn}

\cortext[cor1]{Corresponding author}

\begin{abstract}
We consider the semi-infinite system of polynomial inequalities of the
form
\[
\K:=\{x\in\RR^m\mid p(x,y)\ge 0,\ \ \forall y\in S\subseteq\RR^n\},
\]
where $p(x,y)$ is a real polynomial in the variables $x$ and the
parameters $y$, the index set $S$ is a basic semialgebraic set in
$\RR^n$, $-p(x,y)$ is convex in $x$ for every $y\in S$. 
We propose a procedure to
construct approximate semidefinite representations of $\K$. 
There are two indices to index these approximate semidefinite
representations.
As two indices increase, these semidefinite representation sets expand
and contract, respectively, and can approximate $\K$ as closely as
possible under some assumptions. 
In some special cases,
we can fix one of the two indices or both.
Then, we consider the optimization problem of minimizing a convex
polynomial over $\K$. We present an SDP
relaxation method for this optimization problem by similar strategies used in
constructing approximate semidefinite
representations of $\K$. Under certain assumptions, some approximate
minimizers of the optimization problem can also be obtained from the SDP
relaxations. In some special cases, we show that the SDP relaxation
for the optimization problem is exact and all minimizers can be extracted. 
\end{abstract}



\end{frontmatter}

%
%

\section{Introduction}
We consider the following semi-infinite system of polynomial inequalities
\begin{equation}\label{eq::K}
\K:=\{x\in\RR^m\mid p(x,y)\ge 0,\ \ \forall y\in S\subseteq\RR^n\},
\end{equation}
where $p(x,y)\in\RR[x,y]:=\RR[x_1,\ldots,x_m,y_1,\ldots,y_n]$ the
polynomial ring in $x$ and $y$ over the real field and the {\itshape
index set} $S$ is a basic semialgebraic set defined by
\begin{equation}
	S:=\{y\in\RR^n\mid g_1(y)\ge 0, \ldots, g_s(y)\ge 0\},
	\label{eq::S}
\end{equation}
where $g_j(y)\in\RR[y]$, $j=1,\dots,s$.
In this paper, we assume that $-p(x,y)\in\RR[x]$ is convex for every
$y\in S$ and hence $\K$ is a convex set in $\RR^m$. 

We say a convex set $C$ in
$\RR^m$ is {\itshape semidefinitely representable} (or {\itshape
linear matrices inequality representable}) if there exist some integers
$l, k$ and real $k\times k$ symmetric matrices $\{A_i\}_{i=0}^m$ and
$\{B_j\}_{j=1}^l$ such that $C$ is identical with 
\begin{equation}\label{eq::sdr}
	\left\{x\in\RR^m\ \Big|\ \exists w\in\RR^l,\ \text{s.t.}\ A_0+\sum_{i=1}^m
	A_ix_i+\sum_{j=1}^l B_jw_j\succeq 0\right\}
\end{equation}
and $(\ref{eq::sdr})$ is called the {\itshape semidefinite
representation} (or {\itshape
linear matrices inequality representation}) of $C$. Many interesting
convex sets are semidefinitely representable, see a collection in
\cite{LecturesConOpt}. Clearly,
optimizing a linear function over a semidefinitely representable set
can be cast as a semidefinite progamming (SDP) problem, while SDP has an
extremely wide area of applications and can be efficiently solved by
interior-point methods to a given accuracy in polynomial time (c.f. 
\cite{handbookSDP}). Semidefinite representations of convex sets can
help us to build SDP relaxations of many computationally intractable
optimization problems. Arising from above, 
one of the basic issues in convex algebraic geometry is to
characterize convex sets in $\RR^m$ which are semidefinitely
representable and give systematic procedures to obtain their
semidefinite representations. 
Clearly, if a set in $\RR^m$ is semidefinitely representable, then it
is convex and semialgebraic. Conversely, Nemirovski asked in his plenary address at
the 2006 ICM that whether each convex semialgebraic set is
semidefinitely representable. Yet a negative answer has been
recently given by 
\cite{SSS2018}. Hence, it is reasonable to
study how to construct approximate semidefinite representations of $C$, that is a
sequence of semidefinite representation sets of the form $(\ref{eq::sdr})$ which
converge to $C$ in some sence. 

For a given basic semialgebraic set in $\RR^m$, 
\cite{convexsetLasserre} and 
\cite{thetabody} proposed
some methods to construct semidefinite outer approximations of the
closure of its convex hull. These appproaches are based on the sums of
squares representation of linear functions which are nonnegative on
a basic semialgebraic set. If the basic semialgebraic set is
compact, these approximations can be made arbitrarily close and become
exact under some favorable conditions. Some extensions of these
semidefinite approximations to noncompact basic semialgebraic sets are
given in \cite{SIAMGWZ}. For a convex semialgebraic set, 
\cite{HeltonNie, SRCSHN} proposed some sufficient conditions, in terms
of curvature conditions for the boundary, for its semidefinite
representability. These conditions are recently modified and improved
by 
\cite{KS2018}. 
In this paper, we first consider to
construct approximate semidefinite representations of the set $\K$ in
(\ref{eq::K}). 
The difference of this problem from ones in the literature is that
$\K$ is defined by {\itshape infinitely} many convex real polynomials. 
As there is a quantifier in the definition (\ref{eq::K}), $\K$ is in
fact a semialgebraic set by the Tarski-Seidenberg principle
(c.f. \cite{realAG}). Theoretically, $\K$ can be decomposed as a finite union
of basic closed semialgebraic sets and hence, as proved in
\cite{HeltonNie}, the
semidefinite approximations of $\K$ can be made by glueing together
\cite{convexsetLasserre} relaxations of many small pieces of
$\K$. Such a decomposition of $\K$ can possibly be obtained by
quantifier elimination with algebraic techniques (c.f. \cite{realAG}). 
However, in practice (exact) quantifier elimination is very costly and
limited to problems of very modest size.
These obstacles make the problem studied in this paper nontrivial. 
To the best of our knowledge, there is very few related work in the
literature addressing this issue. 
In \cite{Lasserre15,MHL15}, some tractable methods using
	semidefinite programs are proposed to approximate
	semi-algebraic sets defined with quantifiers. Clearly, the set
	$\K$ studied in this paper is in such case with a universal
	quantifiers. However, rather than approximate semidefinite
	representations of $\K$, their approach generates a sequence
	of sublevel sets of a single polynomial to approximate $\K$.

In the second part of this paper, we consider the following convex 
minimization problem 
	\[
(\P)\qquad	f^*:=\inf_{x\in\K} f(x)\quad\text{where $\K$ is defined in
		(\ref{eq::K}) and $f(x)\in\RR[x]$ is convex.}
	\]
This problem is NP-hard.  Indeed, it is obvious that the problem of
minimizing a polynomial $h(y)\in\RR[y]$ over $S$ can be regarded
as a special case of $(\P)$. 
As is well known, the polynomial optimization problem
is NP-hard even when $n>1$, $h(y)$ is a nonconvex
quadratic polynomial and $g_j(y)$'s are linear (c.f.
\cite{PardalosVavasis1991}).
Hence, a general the problem $(\P)$ cannot
be expected to be solved in polynomial time unless P=NP.

The problem $(\P)$ can be seen as a special branch of
convex {\itshape semi-infinite programming} (SIP), in which the
involved functions are not necessarily polynomials.
Numerically, SIP problems can be solved by different approaches
including, for instance, discretization methods,
local reduction methods, exchange methods,
simplex-like methods and so on.  See \cite{Hettich1993,Still2007,Goberna2017}
and the references therein for details.
One of main difficulties in numerical treatment of general SIP problems
is that the feasibility test of $\bar{u}\in\RR^m$
is equivalent to {\itshape globally} solve the lower level subproblem
of $\min_{y\in S} p(\bar{u},y)$ which is generally nonlinear and
nonconvex.  To the best of our knowledge, few of the numerical methods
mentioned above are specially designed by exploiting features of
polynomial optimization problems.
\cite{PR2009} proposed a
discretization-like method to solve minimax polynomial optimization
problems, which can be reformulated as semi-infinite polynomial
programming (SIPP) problems.
Using polynomial approximation and an appropriate hierarchy of SDP
relaxations, Lasserre presented an
algorithm to solve the generalized SIPP problems in
\cite{Lasserre2011}. Based on an exchange scheme, an SDP
relaxation method for solving
SIPP problems was proposed in \cite{SIPPSDP}.
By using representations of nonnegative polynomials in the
univariate case, an SDP method was given in \cite{XSQ}
for linear SIPP problems (a special case of $(\P)$) with $S$
being closed intervals. 

\vskip 5pt 

Here are some contributions and novelties in this paper: 
\begin{enumerate}[(i)]
	\item We first propose a procedure to
construct approximate semidefinite representations of $\K$ (Section
\ref{subsec::ASR}). 
The construction is based on some representations of linear functions
nonnegative on $\K$. On the one hand, we use high degree perturbation
proposed in \cite{Lasserre05} to approximate the Lagrangian associated
with the considered linear function by sums of squares of polynomials.
As there is an integration with respect to some unknown measure in the
Lagrangian, on the other hand, we employ Putinar's Positivstellensatz
to replace the integration by some linear functionals in the dual
spaces of quadratic modules. Consequently, some semidefinite representation
sets with two indices are obtained to approximate $\K$. 
As two indices increase, these semidefinite representation sets expand
and contract, respectively, and can approximate $\K$ as closely as
possible under some assumptions (Theorem \ref{th::main}). 
In some special cases when we can fix one of the two indices or both
(Remark \ref{rk::simplification}).
\item 
As the second contribution in this paper, we present some new SDP
relaxation methods for the problem $(\P)$ by similar
strategies used in constructing approximate semidefinite
representations of $\K$. Approximate values of $f^*$ can be obtaind
by the proposed SDP relaxations with two indices and converge to
$f^*$ as the two indices tend to $\infty$ (Theorem \ref{th::approf*}).
If $(\P)$ has a unique minimizer, approximate minimizers of $(\P)$ can
also be obtained from the SDP relaxations (Remark \ref{rk::sdp}). 
Compared with
some existing related work, the convexity in $(\P)$ is well exploited
here and the assumptions needed are quite
mild. In the case when $f$ and $-p(x,y)$ are s.o.s-convex for every
$y\in S$, the indices in the SDP relaxations can be reduced to one. 
If, moreover, $S$ is a bounded interval, we show that the SDP
relaxation of $(\P)$ is exact and all minimizers can be extracted
(Theorem \ref{th::mainsdp2}).  
\end{enumerate}

\vskip 5pt
This paper is organized as follows. In Section \ref{sec::pre}, we give
some notation and preliminaries used in this paper. Approximate
semidefinite representations of $\K$ as well as some examples are
proposed in Section \ref{section::SDPrelax}. We study SDP relaxations of
the problem $(\P)$ in Section \ref{sec::SDPCSIPP}.

\section{Notation and Preliminaries}\label{sec::pre}

Here is some notation used in this paper.
The symbol $\N$ (resp., $\re$) denotes
the set of nonnegative integers (resp., real
numbers). For any $t\in \re$, $\lceil t\rceil$ (resp. $\lfloor t\rfloor$)
denotes the smallest (resp. largest) integer that is not smaller
(resp. larger) than $t$.
For $x=(x_1,\ldots,x_m)\in \re^m$,
$\Vert x\Vert_2$ denotes the standard Euclidean norm
of $x$.
For $\alpha=(\alpha_1,\ldots,\alpha_n)\in\N^n$,
$|\alpha|=\alpha_1+\cdots+\alpha_n$.
For $k\in\N$, denote $\N^n_k=\{\alpha\in\N^n\mid |\alpha|\le
k\}$ and $\vert\N^n_k\vert$ its cardinality.
For variables $x=(x_1,\ldots,x_m)$, $y=(y_1,\ldots,y_n)$ and
$\beta\in\N^m, \af \in \N^n$, $x^\beta$, $y^\af$ denote
$x_1^{\beta_1}\cdots x_m^{\beta_m}$, 
$y_1^{\af_1}\cdots y_n^{\af_n}$, respectively.
$\RR[x]$(resp., $\mathbb{R}[y]$) denotes
the ring of polynomials in $x$ (resp., $y$) with real
coefficients. 
For $k\in\N$, denote by $\RR[x]_k$ (resp., $\RR[y]_k$) the set of polynomials
in $\RR[x]$ (resp., $\RR[y]$) of total degree up to $k$.
For $A=\RR[x],\ \RR[y],\ \RR[x]_k,\ \RR[y]_k$, denote by $A^*$ the dual
space of linear functionals from $A$ to $\RR$. 
\begin{definition}
	We say that the {\itshape Slater condition} holds for $\K$ if there exists $u\in\K$
	such that $p(u,y)>0$ for all $y\in S$ and the point $u$ is called a
	{\itshape Slater point}. 
\end{definition}
%
\begin{theorem}{\upshape(c.f. \cite{Borwein1981,Levin1969})}\label{th::red}
	Assume that the Slater condition holds for $\K$ and the index set
	$S$ is compact in the problem $(\P)$. Then for any convex $f(x)\in\RR[x]$,
	there exist points $y_1,\ldots,y_l\in S$ with
	$l\le n$ such that $f^*$ is equal to the optimal value of the
	discretization problem
	\begin{equation}\label{eq::dis}
		\min_{x\in\RR^m} f(x)\quad\text{\upshape s.t. } p(x,y_1)\ge
		0,\ldots,p(x,y_l)\ge 0. 
	\end{equation}
\end{theorem}
\begin{cor}\label{cor::red}
	Suppose that the assumptions in Theorem \ref{th::red} hold for
	$(\P)$.
	Then for any convex $f[x]\in\RR[x]$,
	there exist $u\in\RR^m$, $y_1,\ldots,y_l\in S$ and nonnegative Lagrange
	multipliers $\lambda_1,\ldots,\lambda_l\in\RR$ with $l\le n$ such that
	the Lagrangian 
	\begin{equation}\label{eq::lag}
		L_f(x):=f(x)-f^*-\sum_{i=1}^l\lambda_i p(x,y_i)
\end{equation}
satisfies that $L_f(x)\ge L_f(u)=0$ for all $x\in\RR^m$ and $\nabla
L_f(u)=0$.
\end{cor}
\begin{proof}
	Consider the discretization problem \eqref{eq::dis}. As the Slater
	condition holds for \eqref{eq::dis}, by convex programming duality
	(c.f. \cite[Proposition 5.3.1]{COT}), there is no dual gap between
	\eqref{eq::dis} with its Lagrange dual problem, which has an
	optimal solution, say $\lambda=(\lambda_1,\ldots,\lambda_l)$ where
	$\lambda_i\ge 0$. By a Frank-Wolfe type theorem proved in
	\cite{Belousov1977}, \eqref{eq::dis} also has an optimal solution,
	say $u$. Then, due to convex programming  optimality conditions
	(c.f. \cite[Proposition 5.3.2]{COT}), we get 
	\[
		f(x)-\sum_{i=1}^l\lambda_i p(x,y_i)\ge
		f(u)-\sum_{i=1}^l\lambda_i p(u,y_i),\ \ \forall
		x\in\RR^m\quad\text{and}\quad
		\lambda_ip(u,y_i)=0,\ \ i=1,\ldots, l,
	\]
	which implies that $L_f(x)\ge L_f(u)=0$ for all $x\in\RR^m$ and
	hence $\nabla L_f(u)=0$. 
\end{proof}

\vskip 8pt 

Next we recall some background about representations of polynomials
positive (nonnegative) on a basic semialgebraic set
and the dual theory.
A polynomial $\phi(x)\in\RR[x]$ is said to be a
{\itshape sum of squares} (s.o.s) of polynomials if it can be written as $\phi(x)=\sum_{i=1}^t
\phi_i(x)^2$ for some $\phi_1(x),\ldots,\phi_t(x)\in\RR[x]$.  
Notice that not every nonnegative polynomials can be written as s.o.s,
see \cite{Reznick96someconcrete}. 
\cite{Lasserre05} gave the following s.o.s approximations of
nonnegative polynomials via simple high degree perturbations.
\begin{theorem}{\upshape \cite[c.f. Theorem 3.1, 3.2 and Corollary
	3.3]{Lasserre05}}\label{th::perturbation}
	For a given $h\in\RR[x]$, the following are true.
	\begin{enumerate}[\upshape (i)]
		\item For any $r\ge\lceil\deg(h)/2\rceil$, there exists
			$\varepsilon^*_r\ge 0$ such that
			$h+\varepsilon(1+\sum_{j=1}^m x_j^{2r})$ is s.o.s if and
			only if $\varepsilon\ge\varepsilon^*_r$; 
		\item If $h$ is nonnegative on $[-1,1]^m$, then
			$\varepsilon^*_r$ in $(i)$ decreasingly converges to
			$0$ as $r$ tends to $\infty$;
		\item For any $\varepsilon>0$, if $h$ is nonnegative on
			$[-1,1]^m$, then there exists some
			$r(h,\varepsilon)\in\N$ such that
			$h+\varepsilon(1+\sum_{j=1}^m x_j^{2r})$ is s.o.s for
			every $r\ge r(h,\varepsilon)$. 
	\end{enumerate}
\end{theorem}

Moreover, $\varepsilon^*_r$ in Theorem \ref{th::perturbation} is
computable by solving an SDP problem, see \cite[Theorem 3.1]{Lasserre05}. 
\vskip 5pt 

In the rest of this paper, we let $G:=\{g_1,\ldots,g_s\}$ be the set
of polynomials that defines the semialgebraic set $S$ in
$(\ref{eq::S})$ and $g_0=1$ for convenience.  

We denote by 
\[
  \qm(G):=\left\{\sum_{j=0}^sg_jq_j^2\ \Big|\ 
q_j \in \RR[y], j=0,1,\ldots,s \right\}
\]
the {\itshape quadratic module} generated by $G$ and denote by 
\[
	\qm_k(G):=\left\{\sum_{j=0}^s g_jq_j^2\ \Big|\ 
q_j \in \RR[y], \,
\deg(g_jq_j^2)\le 2k, j=0,1,\ldots,s\right\}
\]
its  $k$-th {\itshape quadratic module}.
It is clear that if $\psi\in\qm(G)$, then
$\psi(y)\ge 0$ for any $y\in S$. However, the converse is not necessarily
true.
Note that checking $\psi\in\qm_k(G)$ for a fixed $k\in\N$ is an
SDP feasibility problem, see \cite{LasserreGlobal2001,parriloSturmfels}.

\begin{definition}\label{def::AC}
We say that ${\cal Q}(G)$
is {\itshape Archimedean }  if
there exists $\psi\in {\cal Q}(G)$ such that the inequality $\psi(y)\ge 0$
defines a compact set in $\RR^n$.
\end{definition}

Note that the  Archimedean property implies that $S$ is compact
but the converse is not
necessarily true. However, for any compact set $S$ we can always
force the associated quadratic module to be Archimedean
by adding a
redundant constraint $M-\Vert y\Vert^2_2\ge 0$ in the description of
$S$ for sufficiently large $M$.

\begin{theorem}\label{th::PP}{\upshape\cite[{Putinar's Positivstellensatz\/}]{Putinar1993}}
Suppose that $\qm(G)$ is Archimedean. If a polynomial $\psi\in\RR[y]$ is
		positive on $S$, then $\psi\in\qm_k(G)$ for some $k\in\N$. 
\end{theorem}
%

For a polynomial $\psi(y)=\sum_\alpha \psi_\alpha y^\alpha\in\RR[y]$, define
the norm
\begin{equation}\label{eq::fnorm}
	\Vert \psi\Vert:=\max_\alpha\frac{\vert
	\psi_\alpha\vert}{\tbinom{|\alpha|}{\alpha}}.
\end{equation}
We have the following result for an estimation of the order $k$ in
Theorem \ref{th::PP}. 
\begin{theorem}{\upshape \cite[Theorem 6]{NieSchweighofer}}\label{th::complexity}
Suppose that $\qm(G)$ is Archimedean and $S\subseteq
(-\tau_S,\tau_S)^n$ for some $\tau_S>0$. Then there is some
positive $c\in\RR$ (depending only on $g_j$'s) such that for all
$\psi\in\RR[y]$ of degree $d$ with $\min_{y\in S}\psi(y)>0$, we have
$\psi\in\qm_k(G)$ whenever
\[
	k\ge c\exp\left[\left(d^2n^d\frac{\Vert \psi\Vert\tau_S^d}{\min_{y\in
	S}\psi(y)}\right)^c\right]. 
\]
\end{theorem}
\vskip 5pt

We say that a linear functional $\mH\in(\RR[y])^*$ has a {\itshape
representing measure} $\mu$ if 
there exists a Borel measure $\mu$ on $\RR^n$ such that
\[
	\mH(y^\alpha)=\int y^{\alpha}\ud\mu(y),\quad \forall \alpha\in\N^n.
\]
For $k\in\N$, we say $\mH\in(\RR[y]_k)^*$ has a representing measure
$\mu$ if the above holds for all $\alpha\in\N^n_{k}$. 
%

A basic problem in the theory of moments concerns the characterization
of linear functionals in $(\RR[y])^*$ which have some representing
measure. 
%
\begin{theorem}{\upshape\cite[Theorem 2.1]{BM1984}}\label{th::Berg}
	Let $\mH$ be a linear functional in $(\RR[y])^*$ such that
	$\mH(q^2)\ge 0$ for all $q\in\RR[y]$. If there exist
	$a, c>0$ such that $\vert \mH(y^\alpha)\vert\le ca^{\vert\alpha\vert}$ for
	every $\alpha\in\N^n$, then $\mH$ has exactly one representing
	measure $\mu$ on $\RR^n$ with support contained in
	$[-c,c]^{n}$. 
\end{theorem}

\cite{Haviland1935} proved that $\mH\in(\RR[y])^*$ has a representing
measure $\mu$ supported on $S$ in \eqref{eq::S} if and only if $\mH(h)\ge
0$ for every $h\in\RR[y]$ nonnegative on $S$. Clearly,
\[
(\qm_k(G))^*=\{\mathscr{H}\in(\RR[y]_{2k})^*\mid \mH(g_jq_j^2)\ge 0,\ \forall q_j\in\RR[y],\
		\deg(g_jq_j^2)\le 2k,\ j=0,1,\ldots,s\}.
\]
Hence, in a dual view, Putinar's Positivstellensatz reads 
\begin{theorem}\label{th::PP2}{\upshape\cite[{Putinar's Positivstellensatz\/}]{Putinar1993}}
Suppose that $\qm(G)$ is Archimedean.  If 
$\mH\in(\qm_k(G))^*$ for all $k\in\N$, 
then $\mH$ has a representing measure $\mu$ supported on $S$.
\end{theorem}


Let 
\begin{equation}\label{eq::d}
	d_j:=\lceil\deg(g_j)/2\rceil,\quad \forall\
	j=1,\ldots,s,\quad\text{and}\quad
	d_S:=\max_j d_j.
\end{equation}


For $k\ge d_S$, 
we have the following sufficient condition for a linear functional
$\mH\in(\qm_k(G))^*$ having
representing measure supported on $S$.
Denote by $M_k(\mH)$ the $k$-th {\itshape moment matrix} associated with a linear
functional $\mH\in(\RR[y]_{2k})^*$, which is  
indexed by $\mathbb{N}^n_{k}$, with
$(\alpha,\beta)$-th entry $\mH(y^{\alpha+\beta})$ for $\alpha, \beta \in
\mathbb{N}^n_{k}$.  
\begin{condition}\label{con::extension}
	A linear functional $\mH\in(\RR[y]_{2k})^*$
	satisfies the {\itshape flat extension condition} when 
	\[
		\rank\ M_{k-d_S}(\mH)=\rank\ M_k(\mH). 
	\]
\end{condition}
\begin{theorem}\label{th::extension}{\upshape \cite[Theorem
	1.1]{CFKmoment}}
	Suppose that 
	$\mH\in(\qm_k(G))^*$
	satisfies 
	the flat extension condition with $r:=\rank
	M_k(\mH)$, then $\mH$ has a unique $r$-atomic representing measure
	supported on $S$.
\end{theorem}

To end this section, let us recall a very interesting subclass of
convex polynomials in $\RR[x]$ introduced by 
\cite{SRCSHN}. 
\begin{definition}{\upshape(\cite{SRCSHN})}
	A polynomial $h\in\RR[x]$ is s.o.s-convex if its Hessian $\nabla^2
	h$ is a s.o.s, i.e., there is some integer $r$ and some matrix
	polynomial $H\in\RR[x]^{r\times m}$ such that $\nabla^2
	h(x)=H(x)^TH(x)$.
\end{definition}

While checking the convexity of a polynomial is generally
NP-hard (c.f. \cite{Ahmadi2013}),
s.o.s-convexity can be checked numerically by solving an
SDP, see \cite{SRCSHN}. The following result plays a significant role
in this paper. 
\begin{lemma}{\upshape\cite[Lemma 8]{SRCSHN}}\label{lem::sosconvex}
	Let $h\in\RR[x]$ be s.o.s-convex. If $h(u)=0$ and $\nabla h(u)=0$
	for some $u\in\RR^m$, then $h$ is s.o.s. 
\end{lemma}

\section{Approximate semidefinite representations of $\K$}\label{section::SDPrelax}
As we always assume that the index set $S$ in the
definition of $\K$ is compact in this paper, we first show that a set
$\K$ with a generic noncompact index set $S$ can be converted
into a system with compact index set. 
Hereafter,  by saying that a property holds for a generic index set
$S$, we mean that it holds for $S$ in the following sense.
If we consider the space
of all coefficients of generators $g_j$'s of all possible sets $S$ of form
$(\ref{eq::S})$
in the canonical monomial basis of $\RR[y]_d$ with $d=\max_{j} \deg(g_j)$,
then coefficients of $g_j$'s of those index sets
$S$ such that the property does not hold are in a Zariski
closed set of the space.

\subsection{Noncompact case}
In this subsection, we consider the set $\K$ in (\ref{eq::K})
with noncompact index set $S$.  We used the technique of
homogenization proposed in \cite{SIPPSDP}
to convert a semi-infinite system (\ref{eq::K}) with a generic noncompact
index set into a system with compact index set.  

For a polynomial $g(y)\in\RR[y]$, denote its homogenization by
${g}^{\h}(\td{y})\in\RR[\td{y}]$, where $\td{y}=(y_0, y_1,\ldots,y_n)$,
i.e.,
${g}^{\h}(\td{y})=y_0^{\deg(g)}g(y/y_0)$. For the basic semialgebraic set
$S$ in $(\ref{eq::S})$, define
\begin{equation}\label{eq::Ss}
\begin{aligned}
	\So&:=\{\td{y}\in\RR^{n+1}\mid {g}^{\h}_1(\td{y})\ge 0,\ \ldots,\
	{g}^{\h}_s(\td{y})\ge 0,\ y_0>0,\ \Vert \td{y}\Vert_2^2=1\},\\
	\St&:=\{\td{y}\in\RR^{n+1}\mid {g}^{\h}_1(\td{y})\ge 0,\ \ldots,\
	{g}^{\h}_s(\td{y})\ge 0,\ y_0\ge 0,\ \Vert \td{y}\Vert_2^2=1\}.
\end{aligned}
\end{equation}
\begin{prop}{\upshape\cite[Proposition 4.2]{SIPPSDP}}\label{prop::equivalent}
	For any $g(y)\in\RR[y]$,
$g(y)\ge 0$ on $S$
if and only if ${g}^{\h}(\tilde{y})\ge 0$ on ${\sf closure}(\So)$.
\end{prop}

Let $d_y:=\deg_y(p(x,y))$ and $p^{\h}(x,\td{y})$ be the homogenization of
$p(x,y)$ with respect to the variables $y$. 
It follows that the set $\K$ in $(\ref{eq::K})$ is equivalent to
\[
	\{x\in\RR^m\mid p^{\h}(x,\td{y})\ge 0,\ \ \forall y\in {\sf
	closure}(\So)\}.
\]
Replacing ${\sf closure}(\So)$ by the basic semialgebraic set
$\St$, we get the following set  
\[
	\wt{\K}:=\{x\in\RR^m\mid p^{\h}(x,\td{y})\ge 0,\ \ \forall y\in\St\}.
\]
It is obvious that $\wt{\K}\subseteq\K$ since ${\sf closure}(\So)\subseteq\St$.
\begin{defi}{\upshape(\cite{exactJacNie})}
	$S$ is said to be closed at $\infty$ if ${\sf closure}(\So)=\St$.
\end{defi}
\begin{remark}\label{re::closedness}{\rm
	Clearly, $\K=\wt{\K}$ when $S$ is closed at $\infty$.
Note that not every set $S$ of form $(\ref{eq::S})$ is closed at $\infty$
even when it is compact \cite[Example 5.2]{NieDisNon}.
However, it is shown in \cite[Theorem 4.10]{SIPPSDP} that the
closedness at $\infty$ is a generic property.
It follows that $\K=\wt{\K}$ for generic index sets $S$.
Note that $\So$ depends only on $S$, while $\St$ depends not
only on $S$ but also on the
choice of the inequalities
$g_1(y)\ge 0, \ldots, g_s(y)\ge 0$. In some cases, we can add some
redundant inequalities in the description of $S$
to force it to be closed at $\infty$ (c.f. \cite{SIAMGWZ}).}
\end{remark}

For any polynomial $g(y)\in\RR[y]$,
denote $\hat{g}(y)$ as its homogeneous part
of the highest degree.  Define
\begin{equation}\label{eq::hatS}
	\widehat{S}:=\{y\in\RR^n\mid \hat{g}_1(y)\ge 0,
	\ldots, \hat{g}_s(y)\ge 0, \Vert y\Vert_2^2=1\}.
\end{equation}
In particular,
denote $\hat{p}(x,y)$ as the homogeneous parts of $p(x,y)$ with
respect to $y$ of the highest degree $d_y$.

\begin{definition}\label{def::slater2}
	We say that the extended Slater condition holds for $\K$
	if there exists a point $u\in\RR^m$ of
	$\K$ such that $p(u,y)>0$
	for all $y\in S$ and $\hat{p}(u,y)>0$
	for all $y\in\widehat{S}$. We call $u$ an extended Slater point 
	of $\K$. 
\end{definition}
\begin{prop}\label{prop::eq}
	The Slater condition holds for $\wt{\K}$ 
	if and only if the extended Slater condition holds for 
	$\K$.
\end{prop}
\begin{proof}
	Suppose that $u$ is an extended Slater point of $\K$.
	For any $\td{v}=(v_0,v)\in\St$, we have $v\in\widehat{S}$ if
	$v_0=0$ and $v/v_0\in S$ otherwise. It is straightforward to
	verify that the Slater condition also holds for $\wt{\K}$
	at $u$.

	Suppose that the Slater condition holds for $\wt{\K}$ at
	$u\in\RR^m$.
	For any point $v\in\RR^n$, we have $(0,v)\in\St$ if $v\in\widehat{S}$
	and $\left(\frac{1}{\sqrt{1+\Vert v\Vert_2^2}},
	\frac{v}{\sqrt{1+\Vert v\Vert_2^2}}\right)\in\St$ if $v\in S$.
	Then similarly, it implies that the extended Slater condition 
	holds for $\K$ at $u$.
\end{proof}
\vskip 5pt 
As a result of the above arguments, it is reasonable to consider the
following assumption in the rest of this paper.
\begin{assumption}\label{ass}
	The set $S$ is compact, $-p(x,y)\in\RR[x]$ is convex for any $y\in
	S$ and the Slater condition holds for $\K$.
\end{assumption}

\subsection{Approximate semidefinite representations of
$\K$}\label{subsec::ASR}
We assume that $\K$ in (\ref{eq::K}) is compact and a scalar $\tau_\K$
such that $\Vert x\Vert_2\le\tau_\K$ for any $x\in\K$ is known. 
For $r\in\N$, define
\[
	\Theta_{r}(x)=\sum_{i=1}^m\left(\frac{x_i}{\tau_\K}\right)^{2r}\in\RR[x].
\]
It is clear that $\Theta_r(x)\le 1$ for any $x\in\K$ and $r\in\N$. 
Denote by $\mathbf{B}$ the unit ball in $\RR^m$. 
Recall the notation $d_S$ in (\ref{eq::d}) and let
\[
	d_x=\deg_x(p(x,y)),\quad d_y=\deg_y(p(x,y))\quad\text{and}\quad 
	d_{\K}:=\max\{\lceil d_y/2\rceil, d_S\}. 
\]

For $\mL\in(\RR[x])^*$ (resp., $\mH\in(\RR[y])^*$), denote by
$\mL(p(x,y))$ (resp., $\mH(p(x,y))$) the image of $\mL$ (resp., $\mH$)
on $p(x,y)$ regarded as an element in $\RR[x]$ (resp., $\RR[y]$) with
coefficients in $\RR[y]$ (resp., $\RR[x]$), i.e.,
$\mL(p(x,y))\in\RR[y]$ (resp., $\mH(p(x,y)))\in\RR[x]$). Hence, some  
notation, like $\mH(\mL(p(x,y)))$, should cause no confusion once
the dual spaces where the linear fuctionals $\mL$ and $\mH$ come from
are specified in the context.

\begin{theorem}\label{th::main}
	Suppose that $\K$ is compact. For any integers
	$r\ge\lceil d_x/2\rceil$ and $t\ge d_{\K}$, define
\begin{equation}\label{eq::lambda}
	\Lambda_{r,t}:=\left\{(\mL(x_1),\ldots,\mL(x_m))\in\RR^m:
	\left\{
		\begin{aligned}
			&\mL\in(\RR[x]_{2r})^*,\ 
			\mL(1)=1,\\
			&\mL(q^2)\ge 0, \
			\forall q\in\RR[x]_r,\\
			&\mathscr{L}(\Theta_k)\le 1,\ k=\lceil d_x/2\rceil,\ldots,r,\\
			&\mL(p(x,y))\in\qm_t(G).
		\end{aligned}\right.\right\}.
	\end{equation}
Then, $\Lambda_{r_2,t}\subseteq\Lambda_{r_1,t}$ for any
			$r_2>r_1\ge\lceil d_x/2\rceil$
			and $\Lambda_{r,t_2}\supseteq\Lambda_{r,t_1}$ for any
			$t_2>t_1\ge d_{\K}$.
			If Assumption \ref{ass} holds, then the
			following are true.
	\begin{enumerate}[\upshape (i)]
		\item For any $\varepsilon>0$, there exists an integer 
			$r(\varepsilon)\ge\lceil d_x/2\rceil$ such that for every
			$r\ge r(\varepsilon)$
			and $t\ge d_{\K}$, it holds that 
			$\Lambda_{r,t}\subseteq\K+\varepsilon\mathbf{B}$. 
		    If $\qm(G)$ is Archimedean, then 
			there exists integer $t(\varepsilon)\ge d_{\K}$ such that for
			every $r\ge\lceil d_x/2\rceil$ and $t\ge t(\varepsilon)$, 	
			it holds that
			$\K\subseteq\Lambda_{r,t}+\varepsilon\mathbf{B}$.  		
			Consequently,
			$\Lambda_{r,t}$ converges to $\K$ as $r$ and $t$ both tend
			to $\infty$;
		\item If 
			the Lagrangian $L_f(x)$ as defined in $(\ref{eq::lag})$ is s.o.s for every linear $f\in\RR[x]$, 
			then 
			$\K\supseteq\Lambda_{r,t_2}\supseteq\Lambda_{r,t_1}$
			for any $r\ge\lceil d_x/2\rceil$, $t_2>t_1\ge d_{\K}$.
			For any $\varepsilon>0$, if moreover, $\qm(G)$ is
			Archimedean, then there exists integer
			$t(\varepsilon)\ge d_{\K}$ such that
			$\K\subseteq\Lambda_{r,t}+\varepsilon\mathbf{B}$ for any
			$r\ge\lceil d_x/2\rceil$, $t\ge t(\varepsilon)$.
			Consequently,  
			$\Lambda_{r,t}$ converges to $\K$ as $t$ tends to
			$\infty$ for any $r\ge\lceil d_x/2\rceil$.
	\end{enumerate}
\end{theorem}
\begin{proof}
	For a fixed $x\in\Lambda_{r_2,t}$, there exists
	$\mL\in(\RR[x]_{2r_2})^*$ 
	satisfying conditions in (\ref{eq::lambda}) for $\Lambda_{r_2,t}$. Let 
	$\mL'$ be the restriction of $\mL$ on $\RR[y]_{2r_1}$. 
	Then, it is clear that $\mL'$ satisfies all conditions in
	(\ref{eq::lambda}) for $\Lambda_{r_1,t}$ and thus
	$x\in\Lambda_{r_1,t}$. Similarly, if $x\in\Lambda_{r,t_1}$, then 
	$x\in\Lambda_{r,t_2}$ for any $t_2>t_1\ge d_{\K}$.
	 
	(i).  Fix an $\varepsilon>0$ and a point
	$v\not\in\K+\varepsilon\mathbf{B}$. Now we prove that there is
	some integer $r(\varepsilon)$ that does not depend on $v$ such
	that $v\not\in\Lambda_{r,t}$ for every $r\ge r(\varepsilon)$ and
	$t\ge d_{\K}$, which implies that
	$\Lambda_{r,t}\subseteq\K+\varepsilon\mathbf{B}$. By
	\cite[Lemma 5]{convexsetLasserre}, there exist $a\in\RR^m$ and
	$b=\min_{x\in\K} a^Tx$ statisfying $\Vert a\Vert_2=1$ and $\vert
	b\vert\le\tau_\K$ such that $a^Tx-b\ge 0$ for any $x\in\K$ and
	$a^Tv-b<-\varepsilon$. 
	Consider the optimization problem $\min_{x\in\K} a^Tx-b$. By
	Corollary \ref{cor::red}, the associated Lagrangian
	$L_{a,b}(x):=a^Tx-b-\sum_{j=1}^l\lambda_j p(x,y_l)$ as defined
	in (\ref{eq::lag}) is nonnegative on $\RR^m$ for some
	$y_1,\ldots,y_l\in S$ and nonnegative
	$\lambda_1,\ldots,\lambda_l\in\RR$. In particular, $L_{a,b}$ is nonnegative on
	$[-\tau_{\K},\tau_{\K}]^m$. 
	By Theorem \ref{th::perturbation} (iii), there is some integer
	$r(\varepsilon)\ge\lceil d_x/2\rceil$ such that for any $r\ge
	r(\varepsilon)$, it holds that 
	\begin{equation}\label{eq::ab}
		a^Tx-b+\frac{\varepsilon}{2}(1+\Theta_r)=\td{q}^2+\sum_{j=1}^l \lambda_j p(x,y_j)
	\end{equation}
	for some $\td{q}\in\RR[x]$. As $r\ge
	r(\varepsilon)\ge\lceil d_x/2\rceil$, we have $\deg(\td{q}^2)\le 2r$.
	Now we show that $r(\varepsilon)$ does
	not depend on $v$.  According to \cite[Sec. 3.3]{Lasserre05},
	$r(\varepsilon)$ depends on $\varepsilon$, the dimension $m$ and
	the size of $a$, $b$, $\lambda_j$'s and the coefficients
	$p(x,y_j)$ regarded as polynomials in $\RR[x]$. Fix a Slater point
	$u_0\in\K$, since $a^Tu_0-b-\sum_{j=1}^l\lambda_j p(u_0,y_j)\ge
	0$, as proved in \cite[Lemma 7]{convexsetLasserre}, we have
	\[
		0\le
		\lambda_j\le\frac{a^Tu_0-b}{p(u_0,y_j)}\le\frac{2\tau_\K}{p(u_0,y_j)}
		\le\frac{2\tau_\K}{\min_{j=1,\ldots,l}p(u_0,y_j)}\le\
		\frac{2\tau_\K}{\min_{y\in
		S}p(u_0,y)}\le\frac{2\tau_\K}{p^*_{u_0}},
	\]
	where 
	$p_{u_0}^*:=\min_{y\in S}p(u_0,y)>0$ since $u_0$ is a Slater point
	and $S$ is compact. Write
	$p(x,y_j)=\sum_{\alpha}p_{x,\alpha}(y_j)x^\alpha$, then 
	$p_{x,\alpha}(y_j)\le\max_\alpha\max_{y\in S}p_{x,\alpha}(y)$. 
	Hence, all $a$, $b$, $\lambda_j$'s and $p_{x,\alpha}(y_j)$'s are
	uniformly bounded, which means that $r(\varepsilon)$ does not
	depend on $v$. For any $r\ge r(\varepsilon)$ and $t\ge d_{\K}$, 
	to the contrary, assume that $v\in\Lambda_{r,t}$.
	Then, there exists $\mL$ satisfying the conditions in
	(\ref{eq::lambda}) for $\Lambda_{r,t}$ with
	$\mathscr{L}(x_i)=v_i$. 
	Let $\mu=\sum_{j=1}^l\lambda_j\delta_{y_l}$ where $\delta_{y_l}$
	denotes the Dirac measure at $y_l$. As $\deg(\td{q}^2)\le 2r$,
	it holds that
	\begin{equation}\label{eq::contradiction}
		\begin{aligned}
			0>
			a^Tv-b+\varepsilon&=\mathscr{L}(a^Tx-b)+\varepsilon\\
			&\ge\mathscr{L}(a^Tx-b)+\frac{\varepsilon}{2}\mathscr{L}(1+\Theta_r)\\
			&=\mathscr{L}\left(\td{q}^2+\int_{S}p(x,y)d\mu(y)\right)\\
			&=\mathscr{L}(\td{q}^2)+\int_S\mathscr{L}(x,y)d\mu(y)\ge
			0,\\
		\end{aligned}
	\end{equation}
	which is a contradiction.  Thus, $v\not\in\Lambda_{r,t}$ and
	$\Lambda_{r,t}\subseteq\K+\varepsilon\mathbf{B}$. 
	
	Fix a Slater point $u_0\in\K$. Let
	$u\in\K$ be arbitrary. Now we first prove
	that there exist a point $\bar{u}\in\K$ and an integer $t(\varepsilon)$
	that does not depend on $u$ (in fact, it depends on $\varepsilon, \K, S, u_0,
	p(x,y), g_j$'s) such that $\Vert u-\bar{u}\Vert_2\le\varepsilon$ and
	$\bar{u}\in\Lambda_{r,t}$ for every $r\ge\lceil d_x/2\rceil$ and
	$t\ge t(\varepsilon)$, which implies that
	$\K\subseteq\Lambda_{r,t}+\varepsilon\mathbf{B}$. 
	If $\Vert u-u_0\Vert_2\le\varepsilon$, then let $\bar{u}=u_0$;
	otherwise,
	let $\lambda=\varepsilon/{\Vert u_0-u\Vert_2}$ and
	$\bar{u}=\lambda u_0+(1-\lambda)u$, then we have
	$1>\lambda\ge\frac{\varepsilon}{2\tau_\K}$,
	$\Vert u-\bar{u}\Vert_2=\lambda\Vert u_0-u\Vert_2=\varepsilon$ and  
	\[
		\begin{aligned}
			p(\bar{u},y)&\ge \lambda
			p(u_0,y)+(1-\lambda)p(u,y)&\ &[\text{as $-p(x,y)$ is
				convex in $x$}]\\
			&\ge\lambda p(u_0,y).&\ &[\text{as $u\in\K$}]\\
		\end{aligned}
\]
Let $\kappa(\varepsilon):=\min\{\frac{\varepsilon}{2\tau_\K},1\}$. 
Then, in either case, it follows that
\[
	p(\bar{u},y)\ge\kappa(\varepsilon)p(u_0,y)\ge\kappa(\varepsilon)p_{u_0}^*>0
\]
for any $y\in S$. 
Write $p(\bar{u},y)=\sum_\beta p_{y,\beta}(\bar{u})y^\beta\in\RR[y]$.
Recall the norm defined in (\ref{eq::fnorm}), then 
\[
	\Vert
	p(\bar{u},y)\Vert=\max_\beta\frac {\vert
		p_{y,\beta}(\bar{u})\vert}{\binom{|\beta|}{\beta}}
	\le \max_{\beta}\frac{\max_{x\in\K}
	\vert p_{y,\beta}(x)\vert}{\binom{|\beta|}{\beta}}
	=:N_p. 
\]
As $\K$ is compact, $N_p$ is well-defined. 
Note that $N_p$ does not depend on $u$ but only on $p$ and $\K$. 
By Theorem \ref{th::complexity}, there exists come positive $c$
depending on $g_j$'s such that $p(\bar{u},y)\in\qm_t(G)$ whenever
\[
	t\ge
	c\exp\left[\left(d_y^2n^{d_y}\frac{N_p\tau_S^{d_y}}
	{\kappa(\varepsilon)p_{u_0}^*}\right)^c\right]=:t(\varepsilon).
\]
	For any $r\ge\lceil d_x/2\rceil$, define a linear functional
	$\mL\in(\RR[x]_{2r})^*$ by $\mL(x^\alpha)=\bar{u}^\alpha$ for all
	$\alpha\in\N^m_{2r}$. 
	Then, it is clear that
	$\mathscr{L}(x_i)=\bar{u}_i$ for $i=1,\ldots,m$, 
	$\mathscr{L}(\Theta_k)\le 1$ for $k=\lceil
	d_x/2\rceil,\ldots,r$ and $\mL(q^2)\ge 0$ for all $q\in\RR[x]_r$. 
	We have $\mL(p(x,y))=p(\bar{u},y)$. 
It implies that $\bar{u}\in\Lambda_{r,t}$ and thus
$\K\subseteq\Lambda_{r,t}+\varepsilon\mathbf{B}$ for every $r\ge\lceil
d_x/2\rceil$ and $t\ge t(\varepsilon)$. 

(ii). By (i), we only need to prove $\Lambda_{r,t}\subseteq\K$ for any
$r\ge\lceil d_x/2\rceil$ and $t\ge d_{\K}$. 
Fix a point $v\not\in\K$. By the Separation Theorem of convex
sets, there exist $a\in\RR^m$ and $b\in\RR$ such that $a^Tx-b\ge 0$
for any $x\in\K$ and $a^Tv-b<0$. As proved in (i), there are some
$y_1,\ldots,y_l\in S$ and nonnegative
$\lambda_1,\ldots,\lambda_l\in\RR$ such that
$a^Tx-b-\sum_{j=1}^l\lambda_j p(x,y_l)$ is nonnegative on $\RR^m$. 
Since 
the associated Lagrangian $L_f(x)$ is s.o.s for every linear function $f$, 
we have  
\begin{equation}\label{eq::ab2}
	a^Tx-b=\td{q}^2+\sum_{j=1}^l \lambda_j p(x,y_j)
\end{equation}
for some $\td{q}\in\RR[x]$. To the contrary, assume that
$v\in\Lambda_{r,t}$. Then, there exist $\mL$
satisfying the conditions in (\ref{eq::lambda}) for $\Lambda_{r,t}$. 
Define $\mu$ as in (i). Like in (\ref{eq::contradiction}), we get 
that 
\begin{equation}\label{eq::contradiction2}
		0> a^Tv-b=\mathscr{L}(a^Tx-b)=\mathscr{L}(\td{q}^2)+
		\int_S(\mL(p(x,y)))d\mu(y)\ge 0,
\end{equation}
which is a contradiction.  Thus, $v\not\in\Lambda_{r,t}$ and hence
$\Lambda_{r,t}\subseteq\K$. 
\end{proof}
\begin{remark}\label{rk::simplification}{\rm
	(i). According to the proof, the conclusions (i) and (ii) in Theorem
	\ref{th::main} are still true if we simplify the condtion
	$\mathscr{L}(\Theta_k)\le 1,\ k=\lceil
	d_x/2\rceil,\ldots,r$ in $(\ref{eq::lambda})$ by
	$\mathscr{L}(\Theta_r)\le 1$.
 
	(ii). In practice, we can let $r=t$ in $\Lambda_{r,t}$ and
	approximate $\K$ by one sequence $\{\Lambda_{r,r}\}$. 
	Suppose that $\qm(G)$ is Archimedean,
	then by Theorem \ref{th::main} (i), for any
	$\varepsilon>0$, there exists $r\ge\max\{\lceil d_x/2\rceil,
		d_{\K}\}$ such that
		$\Lambda_{r,r}\subseteq\K+\varepsilon\mathbf{B}$ and
		$\K\subseteq\Lambda_{r,r}+\varepsilon\mathbf{B}$. That is,
		$\{\Lambda_{r,r}\}$ can approximate $\K$ as closely as
	possible as $r$ increases. 

	(iii).  If $S$ is compact but $\qm(G)$ is not Archimedean,
	then the set $\qm_t(G)$ in the definition of $\Lambda_{r,t}$ in
	$(\ref{eq::lambda})$ can be replaced by the $t$-th order {\itshape
	preordering} in Schm{\"u}dgen's representations of polynomials positive
	on $S$ ({\upshape\cite{Schmugen1991}}).
	Moreover, if we have exact representations of polynomials {\itshape nonnegative}
	on $S$ in some cases, we may fix the order $t$ in $\Lambda_{r,t}$
	and only let $r$ increase. Then, a sequence of nested outer approximate
semidefinite representations of $\K$ can be obtained. For instance,
consider the case
\begin{equation}\label{eq::univariate}
	S=[-1,1]=\{y_1\in\RR\mid g_1(y_1):=1-y_1^2\ge 0\}. 
\end{equation}
By the representations of univariate polynomials nonnegative on an
interval (c.f. \cite{PR2000,Laurent_sumsof}), 
we can fix $t=d_{\K}$ and then the sequence $\Lambda_{r,d_{\K}}$ converges to
$\K$ as $r$ tends to $\infty$. We leave the details here to keep the
paper clean.

(iv). If the Lagrangian $L_f(x)$ is s.o.s for every linear
$f\in\RR[x]$, by the proof of Theorem \ref{th::main} (ii), the
condition $\mathscr{L}(\Theta_k)\le 1,\ k=\lceil d_x/2\rceil,\ldots,r$
is redundant and can be removed. 
In general, it may be difficult to check whether or not the Lagrangian
$L_f(x)$ is s.o.s for every linear $f\in\RR[x]$. However, when
$-p(x,y)$ is s.o.s-convex in $x$ for any $y\in S$, by Corollary
\ref{cor::red} and Lemma \ref{lem::sosconvex}, $L_f(x)$ is indeed
s.o.s for any s.o.s-convex $f\in\RR[x]$ (in particular, for every
linear $f\in\RR[x]$). 
 In particular, if $S$ is in the case \eqref{eq::univariate} and
 $-p(x,y)$ is s.o.s-convex in $x$ for any $y\in S$, then we have the
 {\itshape exact} semidefinite representation $\K=\Lambda_{r,t}$ for any $r\ge
 \lceil d_x/2\rceil$ and $t\ge d_\K$. 
}
	$\hfill\square$
\end{remark}

Note that the standard semidefinite representation \eqref{eq::sdr} of
$\Lambda_{r,t}$
can be easily generated using {\sf Yalmip} (\cite{YALMIP}). 
Moreover, for $m=2$ and $3$, we can first generate the form
\eqref{eq::sdr} of $\Lambda_{r,t}$ 
and then draw it using the software package {\sf Bermeja}
\cite{Bermeja}.

\begin{example}\label{ex::sdr}
Now we present some illustrating
examples. As we shall see, the approximate semidefinite
representations defined in this section are very tight for some given
sets $\K$.
\begin{enumerate}[(1).]
	\item	\label{ex::K1}
		Consider the polynomial 
		\[
			\begin{aligned}
				\begin{aligned}
					f(x_1,x_2,x_3)=&32x_1^8+118x_1^6x_2^2+40x_1^6x_3^2+25x_1^4x_2^4
					-43x_1^4x_2^2x_3^2-35x_1^4x_3^4+3x_1^2x_2^4x_3^2\\
					&-16x_1^2x_2^2x_3^4+24x_1^2x_3^6+16x_2^8+44x_2^6x_3^2+70x_2^4x_3^4+60x_2^2x_3^6+30x_3^8. 
				\end{aligned}
			\end{aligned}
		\]
		It is proved in {\upshape\cite{Ahmadi2012}} that
		$f(x_1,x_2,1)\in\RR[x_1,x_2]$ is convex but 
		not s.o.s-convex. Rotate the shape in the $(x_1,x_2)$-plane defined by 
		$f(x_1,x_2,1)\le 100$ continuously around the origin by $90^\circ$ clockwise. 
		Denote by $\K$ the common area of these shapes in this process.
		We illustrate $\K$ in the left of Figure \ref{fig::nsos}.
		In other words, the set $\K$ is defined by
		\[
			\K=\{(x_1,x_2)\in\RR^2\mid p(x_1,x_2,y_1,y_2)\ge 0,\quad
			\forall y\in S\},
		\]
		where $p(x_1,x_2,y_1,y_2)=100-f(y_1x_1-y_2x_2,y_2x_1+y_1x_2,1)$
		and 
		\[
			S=\{(y_1,y_2)\in\RR^2\mid y_1\ge 0,\ y_2\ge 0,\ y_1^2+y_2^2=1\}.
		\]
		It is clear that the assumptions in Theorem \ref{th::main} holds
		for $\K$ and $d_x=d_y=8$, $d_{\K}=4$. 
		By the software {\sf Bermeja}, the semidefinite representation set
		$\Lambda_{4,4}$ as defined in $(\ref{eq::lambda})$ is
		drawn in gray bounded by the red curve in the right of Figure 
		\ref{fig::nsos}. 
		\begin{figure}
			\centering
			\caption{\label{fig::nsos} The set $\K$ (left) and the
			semidefinite representation set 
			$\Lambda_{4,4}$ (right) in Example \ref{ex::sdr}
			(\ref{ex::K1}).}
			\scalebox{0.5}{
				\includegraphics[trim=150 220 150 220,clip]{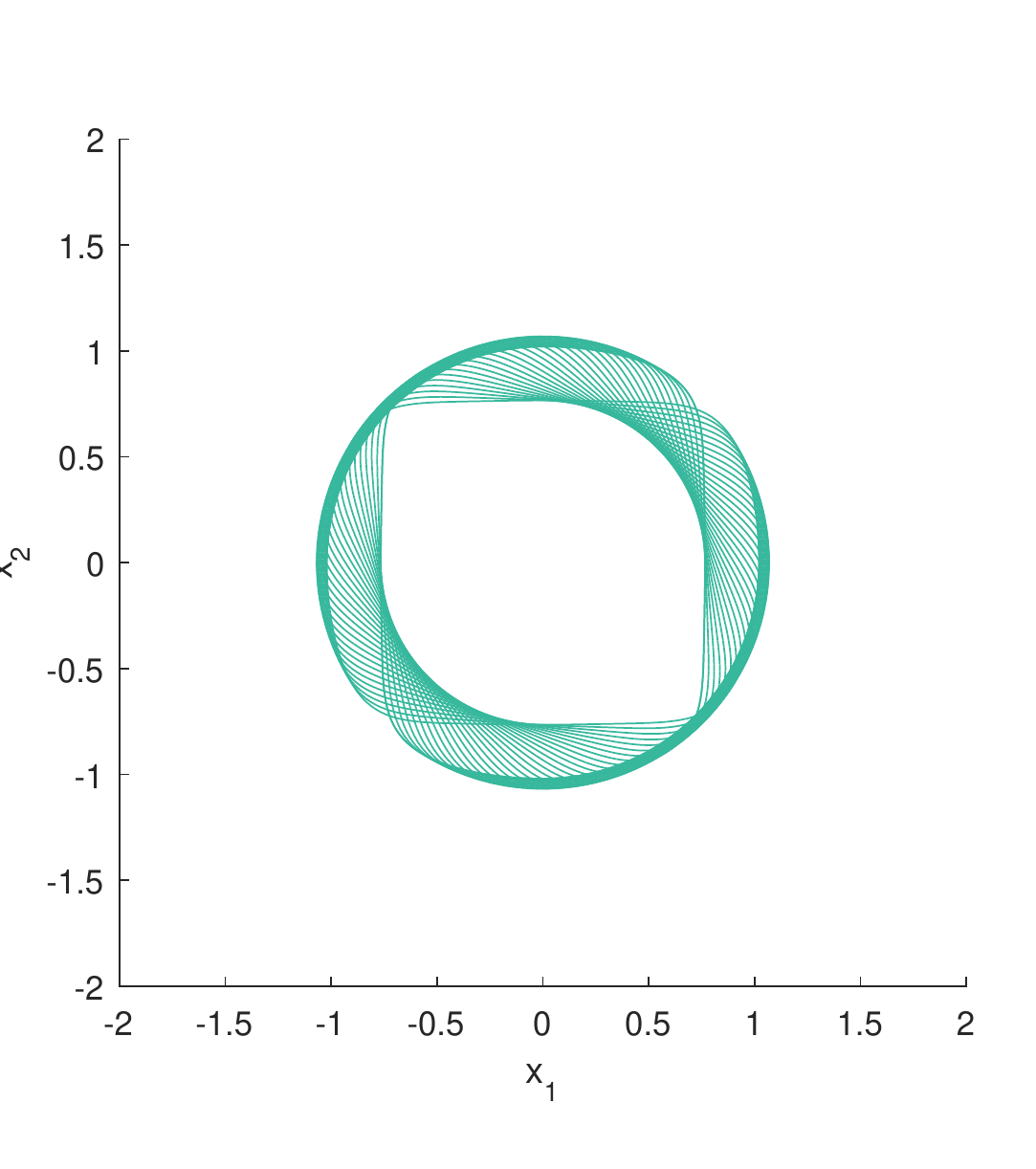}
				\includegraphics[trim=150 220 150 220,clip]{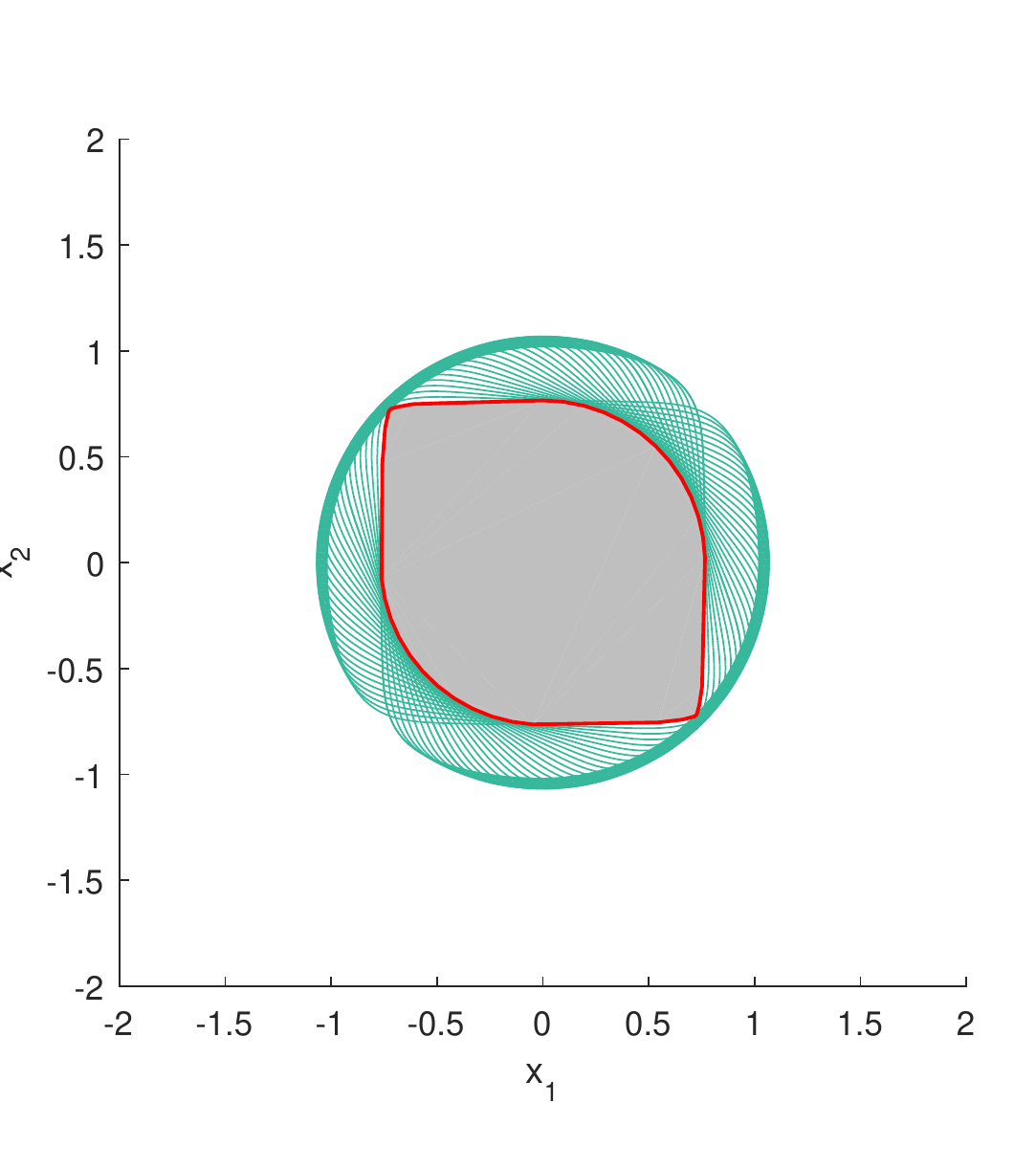}
			}
		\end{figure}
	\item \label{ex::K2}
	Consider the set 
	\[
		\K=\{(x_1,x_2)\in\RR^2\mid p(x_1,x_2,y_1,y_2)\ge
		0,\quad\forall y\in S\}
	\]
	where $p(x_1,x_2,y_1,y_2)=-x_1^2-2y_2x_1x_2-y_1x_2^2-x_1-x_2$ and 
	\[
		S=\{(y_1,y_2)\in\RR^2\mid 1-y_1\ge 0,\ 1/2\ge y_2\ge -1/2,\
		y_1-y_2^2\ge 0\}. 
	\]
	We illustrate $\K$ in the left of Figure \ref{fig::sos} by using
	some grid of $S$. The Hessian matrix of $p$ with respect to $x_1$
	and $x_2$ is 
	\[
		H=-\left[\begin{array}{cc}
		2& 2y_2\\
2y_2& 2y_1
\end{array}\right]\qquad\text{with}\quad \det(H)=4y_2^2-4y_1.
	\]
	Clearly, $-p(x_1,x_2,y_1,y_2)$ is s.o.s-convex in $(x_1,x_2)$
	for every $y\in S$. We have $d_x=2, d_y=1$ and $d_{\K}=1$. 
The 
semidefinite representation set $\Lambda_{1,1}$ is
	drawn in gray bounded by the red curve in the right of Figure 
	\ref{fig::sos}.
\begin{figure}
	\centering
	\caption{\label{fig::sos} The set $\K$ (left) and the semidefinite
	representation set $\Lambda_{1,1}$ (right) in Example \ref{ex::sdr}
	(\ref{ex::K2}).}
\scalebox{0.5}{
	\includegraphics[trim=150 220 150 220,clip]{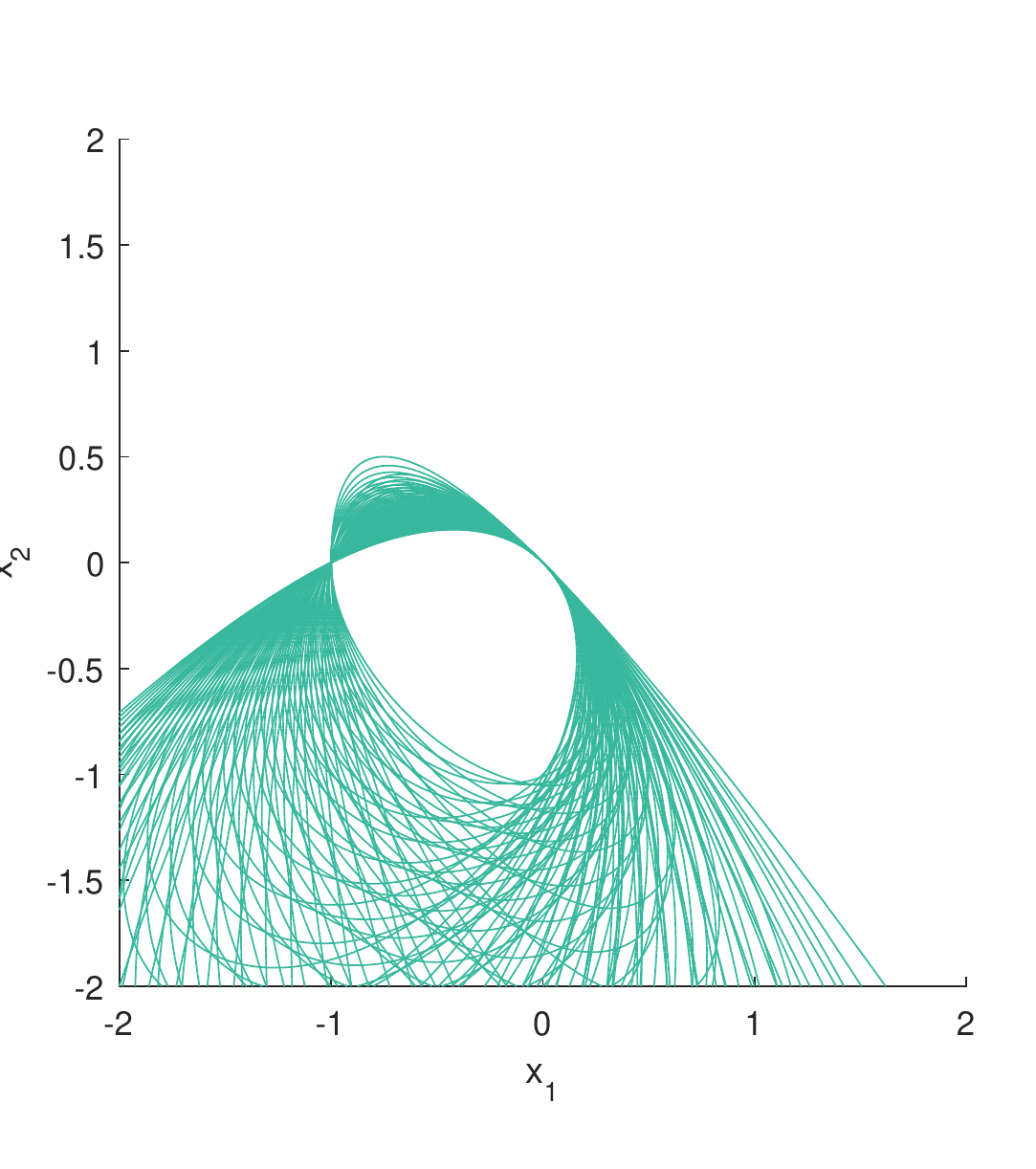}
\includegraphics[trim=150 220 150 220,clip]{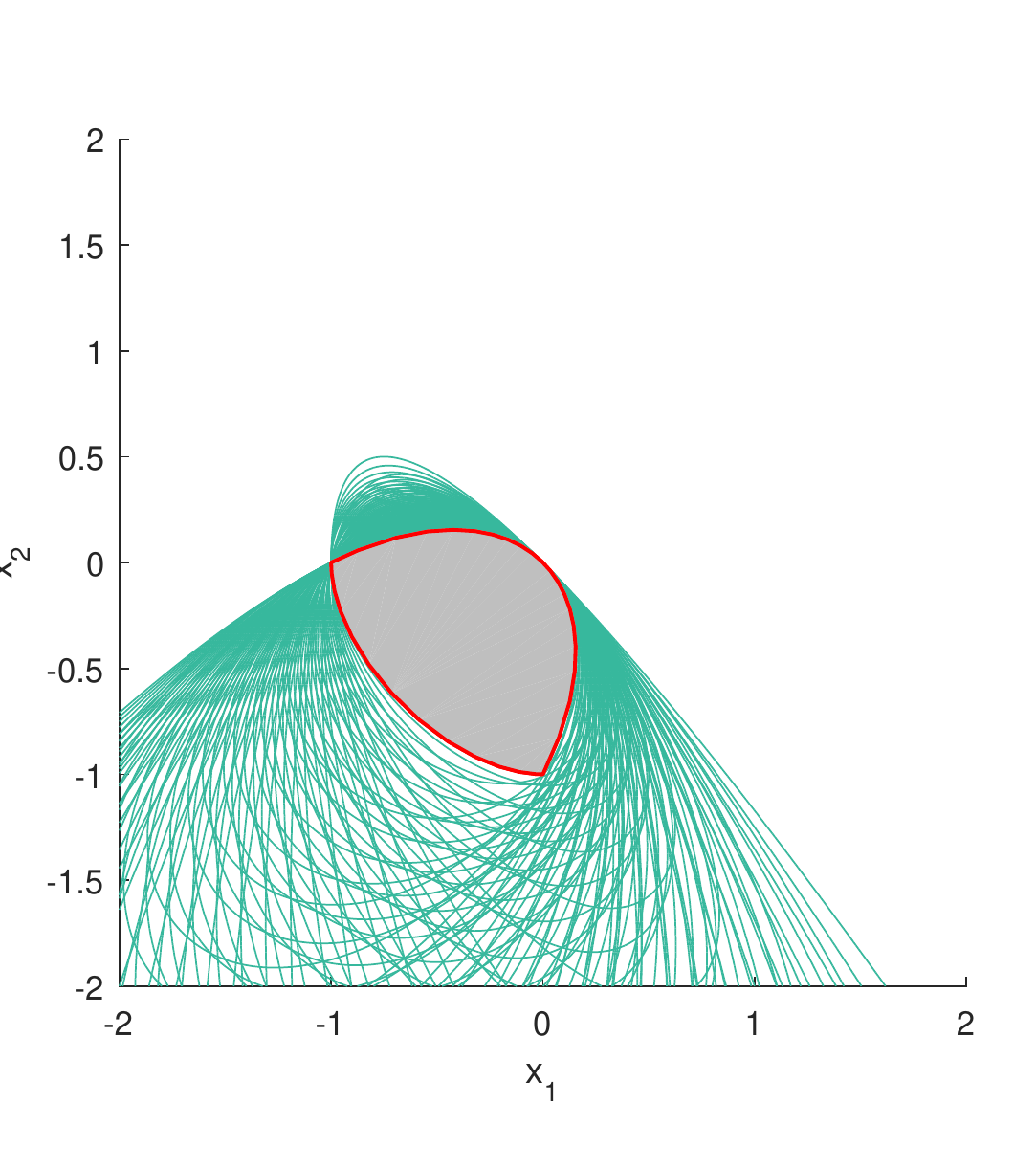}
}
\end{figure}

\end{enumerate}
\end{example}

%
%

\subsection{More discussions}

Now we would like to interpret the semidefinite approximations
$\Lambda_{r,t}$	for $\K$ in a dual view. We shall explain why these
semidefinite approximations need two indices and whether or not we can
approximate the convex hull of $\K$ in a similar way if the convexity
in $x$ is removed
from the constraints functions $-p(x,y)$ for $y\in S$. 

It is clear that the convex hull of a subset in $\RR^m$ is the
intersection of half spaces defined by hyperplanes tangent to this
subset. Hence, to obtain semidefinite approximations of the convex
hull of a subset in $\RR^m$, it is key to characterize linear functions
nonnegative on the subset via s.o.s of polynomials. If
$\K$ is defined by finitely many polynomial inequalities, a linear
function $a^Tx+b$ nonnegative on $\K$ can be represented by Putinar's
(or Schm{\"u}dgen's) Positivstellensatz and convergent semidefinite
approximations of $\K$ can be derived by increasing the degrees of
s.o.s of polynomials invloved in the representation, see
\cite{convexsetLasserre,thetabody}. However, as the set $\K$ in our
case is defined by infinitely many polynomial inqualities, the
Positivstellensatz can not be directly used here. Nevertheless, when $-p(x,y)$
is convex in $x$ for all $y\in S$ and the assumptions in Corollary
\ref{cor::red} hold, there exists a (atomic) measure $\mu$
supported on $S$ for each $a^Tx+b$ such that the associated Lagrangian
$L_{a,b}(x)=a^Tx+b-\int
p(x,y)\ud \mu(y)\ge 0$ on $\RR^m$. Then, to obtain semidefinite
approximations of $\K$, we can use Lasserre's s.o.s representation via
high degree perturbations (Theorem
\ref{th::perturbation}) to characterize
this inequality and the dual of Putinar's Positivstellensatz (Theorem
\ref{th::PP2}) to
replace the unknown measure $\mu$ by a linear functional in
$(\qm_t(G))^*$. Consequently, in the dual, the resulting semdefinite approximations
$\Lambda_{r,t}$ are defined in the way \eqref{eq::lambda} and need two
indices, i.e., one to bound the degree of the perturbation and the
other to bound the order of the quadratic module.

From the above arguments, we can
also see that if the convexity in $x$ is removed from $-p(x,y)$ for
$y\in S$, the convex hull of $\K$ can not be approximated as closely
as possible in a way similarly as $\Lambda_{r,t}$ is defined.
To see this, recall that even in the finitely many
constraints case mentioned above, one need to increase the degree of
s.o.s of polynomials involved in the
Putinar's (or Schm{\"u}dgen's) representation of $a^Tx+b$ to obtain
convergent semidefinite approximations. In the infinitely many
constraints case, to formulate the nonnegative Lagrangian $L_{a,b}$,
we need a measure $\mu$ for $a^Tx+b$ to encode those 
active $y\in S$ (Corollary \ref{cor::red}). As $\mu$ is
unknown, we can
not further parameterize unknown s.o.s of polynomials in the integral $\int
p(x,y)\ud\mu(y)$ and increase the degree; otherwise, bi-linearity
occurs and thus semidefinite
approximations can not be derived. Moreover, such a measure $\mu$ for
$a^Tx+b$ may not even exist if the convexity in $x$ is removed from
$-p(x,y)$. Therefore, it is still a challenge to construct semidefinite
approximations of convex hull of semi-algebraic sets defined by infinitely many
arbitrary polynomial inequalities. 
\vskip 5pt
	In \cite{Lasserre15,MHL15}, some tractable methods using
	semidefinite programs are proposed to approximate
	semi-algebraic sets defined with quantifiers. Clearly, the set
	$\K$ studied in this paper is in such case with a universal
	quantifiers. To end this section, we would like to point out the
	differences in methodology and contributions between the present
	paper and the above two references. The following is the basic idea of
	\cite{Lasserre15,MHL15} to get approximations of $\K$ in
	\eqref{eq::K}. For $x\in\RR^m$, define the map
	$J_p(x):=\min_{y\in S} p(x,y)$. Then, we have $\K=\{x\in\RR^m\mid
	J_p(x)\ge 0\}$. Suppose that $\K$ is contained in a compact set
	$\mathscr{B}$ in
	$\RR^m$, then it can be proved that there exsits a sequence of
	polynomials $\{q_k\}_k\subseteq\RR[x]$ such that $q_k(x)\le
	J_p(x)$ for all $x\in\mathscr{B}$ and $q_k$ converges to $J_p$ 
	for the $L_1(\mathscr{B})$-norm. Hence, $\K$ can be approximated
	by $\{x\in\RR^m\mid
	q_k(x)\ge 0\}$. As $p(x,y)-q_k(x)\ge 0$ for every
	$x\in\mathscr{B}$ and $y\in S$, we can use Positivstellensatz in
	$\RR[x,y]$ to reduce the problem of computing such a sequence
	$\{q_k\}_k$ to SDP problems. Therefore, the method in
	\cite{Lasserre15,MHL15} works for $\K$ in a general form without
	requiring $-p(x,y)$ to be
	convex in $x$ and approximates $\K$ by a sequence of sublevel
	set of a single polynomial. Instead, we exploit the convexity of
	the defining polynomials of 
	$\K$ and construct semidefinite approximations for it.
	Note that the polynomials $q_k$'s in method of
	\cite{Lasserre15,MHL15} can be enforced to be convex for $\K$ in
	\eqref{eq::K}
	(see \cite[Section 4.2]{Lasserre15}),
	but the convergence and the semidefinitely representability of the sublevel
	sets are not clear to the best of our knowledge. 

\section{SDP relaxations of convex semi-infinite polynomial
programming}\label{sec::SDPCSIPP}
For a convex polynomial $f(x)\in\RR[x]$, consider the following convex
semi-infinite polynomial programming problem
\[
(\P)\qquad	f^*:=\inf_{x\in\K}\ f(x)\quad\text{where $\K$ is defined in
		(\ref{eq::K})}.
	\]
Let $d_P:=\max\{\deg(f),d_x\}$ and $\mathcal{M}(S)$ be the set of all
(nonnegative) Borel measures supported on $S$.

\subsection{General case}
Consider the case when $\K$ is compact and Assumption
\ref{ass} holds. 
In the following, we will obtain SDP relaxations of $(\P)$ in two
steps.

In the first step, 
for any integer $r\ge\lceil d_P/2\rceil$, we convert
$(\P)$ to the problem
\begin{equation}\label{eq::f*r}
(\P_r)\qquad	\left\{
	\begin{aligned}
		f_r^*:=\sup_{\rho,\eta,\mu,q}\
		&\rho-2\eta&\\
		\text{s.t.}\ &f(x)-\rho+\eta(1+\Theta_r)=\int_{S}
		p(x,y)d\mu(y)+q^2,&\\
		&\rho\in\RR,\ \eta\ge 0,\ \mu\in\mathcal{M}(S),\
		q\in\RR[x]_{r}.&
	\end{aligned}\right.
\end{equation}
For $\mathscr{L}\in(\RR[x]_{2r})^*$, $\xi\ge 0$, $\rho\in\RR$,
$\eta\ge 0$, $\mu\in\mathcal{M}(S)$ and
$q\in\RR[x]_{r}$,
consider the Lagrange dual function of \eqref{eq::f*r}:
\[
	\begin{aligned}
		L(\rho,\eta,\mu,q,\mathscr{L},\xi)
		:=&\rho-2\eta+\mathscr{L}\left(f(x)-\rho+\eta(1+\Theta_r)-\int_{S}
		p(x,y)d\mu(y)-q^2\right)+\xi\eta\\
		=&\mathscr{L}(f)+(1-\mathscr{L}(1))\rho+(\mathscr{L}(\Theta_r)+\mL(1)-2+\xi)\eta-\int_S
		\mathscr{L}(p(x,y))\ud\mu(y)-\mathscr{L}(q^2)\\
	\end{aligned}
\]
Then, 
\[
	\begin{aligned}
&\sup_{\rho,\eta,\mu,q}\
L(\rho,\eta,\mu,q,\mathscr{L},\xi)\quad\text{s.t.}\ \rho\in\RR,\ \eta\ge 0,\ \mu\in\mathcal{M}(S),\
		q\in\RR[x]_{r}\\
		&=\left\{\begin{array}{rll}
			\mathscr{L}(f)&\mbox{if} &\mathscr{L}(1)=1,\
			\mathscr{L}(\Theta_r)+\xi\le 1,\\
			&		&\mathscr{L}(q^2)\ge 0,\ \forall\ q\in\RR[x]_{r},\\
			&				&\mathscr{L}(p(x,y))\ge 0,\ \forall\ y\in S,\\
			+\infty&\mbox{otherwise}.& 
		\end{array}
		\right.
	\end{aligned}
\]
Hence, the Lagrange dual problem of \eqref{eq::f*r} reads
\begin{equation}\label{eq::f*rdual}
(\P^*_r)\qquad	\left\{
	\begin{aligned}
		\inf_{\mL\in(\RR[x]_{2r})^*}\
		&\mathscr{L}(f)&\\
		\text{s.t.}\
		& \mL(1)=1,\ \mathscr{L}(\Theta_r)\le 1,&\\ 
		&\mL(q^2)\ge 0,\
		\forall q\in\RR[x]_r,&\\
		&\mL(p(x,y))\ge 0,\ \forall y\in S.&\\
	\end{aligned}\right.
\end{equation}

\begin{definition}
	We call $\mL^{(r)}\in(\RR[x]_{2r})^* $ with $r\ge\lceil
	d_P/2\rceil$ a nearly optimal solution of
	$(\ref{eq::f*rdual})$ if $\mL^{(r)}$ is feasible for
	$(\ref{eq::f*rdual})$ and 
	the limit of $\mathscr{L}^{(r)}(f)$ is equal to the limit of
	the optimal values of $(\P^*_r)$ as $r\rightarrow\infty$. 
\end{definition}

\begin{theorem}\label{th::mainOP}
	Suppose that $f(x)$ is convex, $\K$ is compact and Assumption \ref{ass} holds.
	Let $\mL^{(r)}$ be a nearly optimal solution of
	$(\ref{eq::f*rdual})$ and
	$\mL^{(r)}(x)=(\mL^{(r)}(x_1),\ldots,\mL^{(r)}(x_m))$.
	\begin{enumerate}[\upshape (i)]
	\item $f^*_r\le f^*$ and $f^*_r$ 
		converges to $f^*$ as $r$ tends to $\infty$;
	\item $f^*_r$ is attainable in $(\ref{eq::f*r})$ and there is no dual gap between
		$(\ref{eq::f*r})$ and $(\ref{eq::f*rdual})$;
	\item Assume that $\tau_\K=1$ $($possibly after scaling$)$. 
		Then, for any convergent subsequence
		$\{\mL^{(r_i)}(x)\}_i$ of $\{\mL^{(r)}(x)\}_r$, 
	$\lim_{i\rightarrow\infty}\mL^{(r_i)}(x)$ is a minimizer of
		$(\P)$. 
		Consequently, if $u^*$ is the unique minimizer of $(\P)$, then 
		$\lim_{r\rightarrow\infty}{\mL^{(r)}(x)}=u^*$;
	\item If moreover, the Lagrangian $L_f(x)$ as defined in
		$(\ref{eq::lag})$ is s.o.s,
		then $f_r^*=f^*$ for any $r\ge\lceil d_P/2\rceil$ and it is
		also attainable in $(\ref{eq::f*rdual})$. 
	\end{enumerate}
\end{theorem}
\begin{proof}
	(i) For any $x\in\K$ and $y\in S$, we have $\Theta_r(x)\le 1$ and $p(x,y)\ge 0$.
	Consequently, for any feasible point $(\rho,\eta,\mu,q)$ of
	(\ref{eq::f*r}) and any $x\in\K$, it holds that
	\[
		\begin{aligned}
			f(x)&=\rho-\eta(1+\Theta_r(x))+\int_S p(x,y)d\mu(y)+q^2\\
			&\ge \rho-\eta(1+\Theta_r(x))\\
			&\ge \rho-2\eta, 
		\end{aligned}
	\]
which implies that $f^*_r\le f^*$. 

Conversely, by Corollary \ref{cor::red},
there exist some $y_1,\ldots,y_l\in S$ 
and nonnegative Lagrange multipliers
$\lambda_1,\ldots,\lambda_l\in\RR$ such that 
\begin{equation}\label{eq::Lag}
	f(x)-f^*-\sum_{j=1}^l\lambda_j p(x,y_l)=
	f(x)-f^*-\int_{S} p(x,y)d\mu(y)
	\ge 0,\quad\forall x\in\RR^m, 
\end{equation}
where $\mu=\sum_{j=1}^l\lambda_j\delta_{y_l}\in\mathcal{M}(S)$ and $\delta_{y_l}$ is
the Dirac measure at $y_l$.
For any fixed $r\in\N$ with $r\ge\lceil d_P/2\rceil$, by Theorem
\ref{th::perturbation} (i),
there exists a $\varepsilon_r^*\ge 0$ such that 
\begin{equation}\label{eq::L}
	f(x)-f^*-\int_S p(x,y)d\mu(y)+\eta(1+\Theta_r) \quad\text{is
	s.o.s in}\ \RR[x]_{2r}
\end{equation}
if and only if $\eta\ge\varepsilon_r^*$. It means that
$(\ref{eq::f*r})$ is feasible and $f^*_r\ge f^*-2\varepsilon^*_r$.  
Moreover,  by Theorem \ref{th::perturbation} (ii), 
$\varepsilon_r^*$ decreasingly converges to $0$ as $r$ tends to
$\infty$. It then follows
that $f_r^*$ converges to $f^*$ as $r$ tends to $\infty$. 

(ii) Fix a Slater point $u$ of $\K$. Since $S$ is compact, there
exists a neighborhood $\mathcal{O}_u$ of $u$ such that every point in
$\mathcal{O}_u$ is a Slater point of $\K$. Let $\nu$ be the probability
measure with uniform distribution in $\mathcal{O}_u$ and set
$\mL\in(\RR[x]_{2r})^*$ where $\mL(x^\alpha)=\int x^\alpha
d\nu$. It is easy to see
that $\mL$ is strictly admissible for (\ref{eq::f*rdual}).
The conclusion follows due to the duality theory in convex optimization. 

(iii) For any $r\ge \lceil d_P/2\rceil$, as $\tau_{\K}=1$ and
$\mathscr{L}^{(r)}(\Theta_r)\le
1$, it is clear that $\mathscr{L}^{(r)}(x_i^{2r})\le 1$ for all $i=1,\ldots,n$. 
Since $\mL(1)=1$ and $\mL(q^2)\ge 0$ for all $q\in\RR[x]_r$, we then
deduce that $\vert \mL^{(r)}(x^\alpha)\vert\le 1$ for any $|\alpha|\le 2r$ by
\cite[Lemma 4.1 and 4.3]{Lasserre05}. 
Extend $\mL^{(r)}\in(\RR[x]_{2r})^*$ to $(\RR[x])^*$ by letting
$\mL^{(r)}(x^\alpha)=0$ for all $\vert\alpha\vert>2r$ and denote it by
$\wt{\mL}^{(r)}$.  
Then, it holds that $\wt{\mL}^{(r)}(x^\alpha)\in[-1,1]$
for all $\alpha\in\N^m$. 

Let $\{\mL^{(r_i)}(x)\}_i$ be a
convergent subsequence of $\{\mL^{(r)}(x)\}_r$. By Tychonoff's theorem,
there exists a convergent subsequence of the corresponding
$\{\wt{\mL}^{(r_i)}(x^\alpha)\mid \alpha\in\N^m\}_i$ in the product topology. 
Without loss of generality, we assume that the whole
sequence $\{\wt{\mL}^{(r_i)}(x^\alpha)\mid \alpha\in\N^m\}_i$
converges as $i\rightarrow\infty$ and
denote by $\wt{\mL}^*\in(\RR[x])^*$ the limit. 
From the pointwise convergence, we have
$\wt{\mathscr{L}}^*(q^2)\ge 0$ for all $q\in\RR[x]$ and
$\wt{\mL}^*(x^\alpha)\in[-1,1]$ for all $\alpha\in\N^m$. 
By Theorem \ref{th::Berg}, $\wt{\mL}^*$ has exactly
one representing measure $\nu$ with support contained in
$[-1,1]^{m}$. Since $\mL^{(r)}$ is nearly optimal solution of
(\ref{eq::f*rdual}), we obtain $\wt{\mL}^*(f)=\int f d\nu(x)=f^*$ by (i) and (ii).
We have  
\[
	\lim_{i\rightarrow\infty}\mL^{(r_i)}(x)=\wt{\mL}^*(x)=\left(\int
	x_1 d\nu(x),\ldots,\int x_m d\nu(x)\right). 
\]
For any $\varepsilon>0$, from the proof of Theorem \ref{th::main} (i) and
Remark \ref{rk::simplification} (i), it is easy to see that there exists
an integer $r(\varepsilon)$ such that
$\mL^{(r_i)}(x)\in\K+\varepsilon\mathbf{B}$ whenever $r_i\ge
r(\varepsilon)$. 
By the pointwise convergence, 
we deduce that $\wt{\mL}^*(x)\in\K$. 
%
Then, since $f$ is convex and $\wt{\mL}^*$ has a representing measure, by
Jensen's inequality, $f^*\le f(\wt{\mL}^*(x))\le\wt{\mL}^*(f)=f^*$.
Hence, $\wt{\mL}^*(x)$ is indeed a minimizer of (\ref{eq::f*rdual}). 

Assume that $u^*$ is the unique minimizer of (\ref{eq::f*rdual}).
We have shown that $\{\mL^{(r)}(x)\}_r$ is contained in
$[-1,1]^{m}$ and
$\lim_{i\rightarrow\infty}{\mL^{(r_i)}(x)}=u^*$ for any convergent
subsequence $\{\mL^{(r_i)}(x)\}_i$, therefore the whole sequence
$\{\mL^{(r)}(x)\}_r$ converges to $u^*$.

(iv) Under the assumption, (\ref{eq::L}) holds for $\eta=0$ and any 
$r\ge\lceil d_P/2\rceil$. Hence, $f_r^*=f^*$ for any $r\ge\lceil
d_P/2\rceil$ by the proof of (i). 
As $\K$ is compact, suppose that $f^*$ is attainable in $(\P)$ at a
minimizer $x^*\in\K$. 
Define $\mL^*\in(\RR[x]_{2r})^*$ by letting $\mL^*(x^\alpha)=(x^*)^\alpha$
for all $\alpha\in\N^m_{2r}$, then $f^*_r=f^*$
is attainable in (\ref{eq::f*rdual}) at $\mL^*$. 
\end{proof}

\vskip 10pt 
Consider the problem $(\P_r)$. The integration  $\int_{S}\cdot
d\mu(y)$ can be seen as a linear functional in $(\RR[y])^*$. In the
second step, to
obtain SDP relaxations of $(\P_r)$, we need to characterize those linear
functionals $\mH\in(\RR[y])^*$ which have representing measures in
$\mathcal{M}(S)$. In a dual view, we need a representation of 
$\mL(p(x,y))\in\RR[y]$ in \eqref{eq::f*rdual} which is nonnegative on
$S$. Here, Putinar Positivstellensatz (Theorem
\ref{th::PP} and \ref{th::PP2}) comes into play.

For any $t\ge d_{\K}$, consider the SDP relaxation of (\ref{eq::f*r})

\begin{equation}\label{eq::f*rt}
	\left\{
	\begin{aligned}
		f_{r,t}^{\psdp}:=\sup_{\rho,\eta,\mH,q}\
		&\rho-2\eta&\\
		\text{s.t.}\
		&f(x)-\rho+\eta(1+\Theta_r)=\mH(p(x,y))+q^2,&\\
		&\rho\in\RR,\ \eta\ge 0,\ \mH\in(\qm_t(G))^*,\ 
		q\in\RR[x]_{r}.&
	\end{aligned}\right.
\end{equation}
Similar to \eqref{eq::f*r}, the Lagrange dual function of \eqref{eq::f*rt}
is 
\[
		L(\rho,\eta,\mH,q,\mathscr{L},\xi)
		=\mathscr{L}(f)+(1-\mathscr{L}(1))\rho+(\mathscr{L}(\Theta_r)+\mL(1)-2+\xi)\eta-
		\mathscr{H}(\mathscr{L}(p(x,y)))-\mathscr{L}(q^2),
\]
where $\mathscr{L}\in(\RR[x]_{2r})^*$, $\xi\ge 0$, $\rho\in\RR$,
$\eta\ge 0$, $\mH\in(\qm_t(G))^*$ and $q\in\RR[x]_{r}$.
Similar to the duality between \eqref{eq::f*r} and \eqref{eq::f*rdual}, 
the Lagrange dual problem of \eqref{eq::f*rt} can be derived as
\begin{equation}\label{eq::f*rtdual}
	\left\{
	\begin{aligned}
		f_{r,t}^{\dsdp}:=\inf_{\mL}\
		&\mathscr{L}(f)&\\
		\text{s.t.}\
		&\mL\in(\RR[x]_{2r})^*,\ \mL(1)=1,\ \mathscr{L}(\Theta_r)\le
		1,&\\
		&\mL(q^2)\ge 0,\ \forall q\in\RR[x]_r,\
		\mL(p(x,y))\in\qm_t(G).&\\
	\end{aligned}\right.
\end{equation}
\begin{theorem}\label{th::mainsdp}
	For any integer $r\ge\lceil d_P/2\rceil$, the following are true. 
	\begin{enumerate}[\upshape (i)]
		\item If $\qm(G)$ is Archimedean and the Slater
			condition holds for $\K$, then 
			$f_{r,t}^{\psdp}$ and $f_{r,t}^{\dsdp}$ decreasingly
			converge to $f_r^*$ as $t$ tends to $\infty$;
		\item For some order $t\ge d_{\K}$, if the flat extension
			condition
			holds for $\mH^*$ in the solution $(\rho^*,\eta^*,\mH^*,q^*)$
			of $(\ref{eq::f*rt})$, then $f_{r,t}^{\psdp}=f_r^*$;
		\item If $S$ is in the case \eqref{eq::univariate}, then 
			we have $f_{r,d_\K}^{\psdp}=f_{r,d_\K}^{\dsdp}=f_r^*$. 
	\end{enumerate}
\end{theorem}
\begin{proof}
	(i) For any feasible point $(\rho,\eta,\mu,q)$ of
	(\ref{eq::f*r}), define $\mH\in(\qm_t(G))^*$ by letting
	$\mH(y^\beta)=\int y^\beta d\mu$ for all $\beta\in\N^n_{2t}$, 
	then $(\rho,\eta,\mH,q)$ is
	feasible for (\ref{eq::f*rt}) and hence $f_{r,t}^{\psdp}\ge f^*_r$
	for any $t\ge d_{\K}$. 
	Then by the weak duality and Theorem \ref{th::mainOP}, we have
	$f^*_r\le f_{r,t}^{\psdp}\le f_{r,t}^{\dsdp}$ for any $t\ge d_{\K}$. 
	It is sufficient to prove that
	$\lim_{t\rightarrow\infty}f_{r,t}^{\dsdp}=f_r^*$. 

	Fixing an arbitrary $\varepsilon>0$, we show that there is some
	$t\ge d_{\K}$ such that $0\le
	f_{r,t}^{\dsdp}-f_r^*\le\varepsilon$. Fix a Slater point $u$ of
	$\K$ and define $\mL'\in(\RR[x]_{2r})^*$ with
	$\mL'(x^\alpha)=u^\alpha$ for all $\alpha\in\N^m_{2r}$. 
	Then $\mL'$ is feasible for (\ref{eq::f*rtdual}) for some $t'\ge
	d_{\K}$ by Putinar's Positivstellensatz. If $\mL'(f)-f^*_r\le\varepsilon$, then
	$0\le f_{r,t'}^{\dsdp}-f_r^*\le\varepsilon$. 
	Next, we assume that $\mL'(f)-f^*_r>\varepsilon$. Then, we can choose another feasible
	point $\overline{\mL}$ of (\ref{eq::f*rdual}) such that
	$\mL'(f)-\overline{\mL}(f)>0$ and
	$\overline{\mL}(f)-f^*_r\le\varepsilon/2$. Let
	\[
		\delta:=\frac{\varepsilon}{2(\mL'(f)-\overline{\mL}(f))}\quad\text{and}\quad 
			\widehat{\mL}=(1-\delta)\overline{\mL}+\delta \mL'. 
	\]
    Then, we have $0<\delta<1$ and hence 
	\[
		\widehat{\mL}(p(x,y))=(1-\delta)\overline{\mL}(p(x,y))+\delta
		\mL'(p(x,y))>0,\quad\forall y\in S. 
	\]
	Hence, $\widehat{\mL}$ is feasible for (\ref{eq::f*rtdual}) for some
	$\hat{t}\ge d_{\K}$ by Putinar's Positivstellensatz. We have
	\[
		\begin{aligned}
			f_{r,\hat{t}}^{\dsdp}-f_r^*&\le\widehat{\mL}(f)-f_r^*\\
			&=(1-\delta)\overline{\mL}(f)+\delta\mL'(f)-f_r^*\\
			&=\overline{\mL}(f)-f_r^*+\delta(\mL'(f)-\overline{\mL}(f))\\
			&\le\frac{\varepsilon}{2}+\frac{\varepsilon}{2}=\varepsilon. 
		\end{aligned}
	\]
	As $\varepsilon$ is arbitrary, the conclusion follows.

	(ii) Suppose that the flat extension condition holds for
	$\mH^*$ in the 
	solution $(\rho^*,\eta^*,\mH^*,q^*)$ of (\ref{eq::f*rt}) at
	some order $t\ge d_{\K}$. 
	Then,
	by Theorem \ref{th::extension}, $\mH^*$ admits some representing
	measure $\mu^*$ supported on $S$.
	As $f_r^*\le f_{r,t}^{\psdp}$ and $(\rho^*,\eta^*,\mu^*,q^*)$
	is feasible for (\ref{eq::f*r}), we conclude that $f_r^*=f_{r,t}^{\psdp}$. 

	(iii) By the proof of (i), the conclusion follows due to 
	the representations of univariate polynomials nonnegative on an interval (c.f.
\cite{PR2000,Laurent_sumsof}) and Theorem \ref{th::mainOP} (ii). 
\end{proof}

\begin{theorem}\label{th::approf*}
	Suppose $f(x)$ is convex , $\K$ is compact and Assumption
	\ref{ass} holds. Then, for any $\varepsilon>0$, the following are
	true.
	\begin{enumerate}[\upshape (i)]
		\item There exists a $r(\varepsilon)\in\N$ such that
			$f^{\dsdp}_{r,t}\ge f^{\psdp}_{r,t}\ge f^*-\varepsilon$
			holds for any $r\ge r(\varepsilon)$ and $t\ge d_{\K}$;
		\item If $\qm(G)$ is Archimedean and the Slater condition
			holds for $\K$, then for any
			$r\ge\lceil d_P/2\rceil$, there exists a
			$t(r,\varepsilon)\in\N$ such that $f^{\psdp}_{r,t}\le
			f^{\dsdp}_{r,t}\le f^*+\varepsilon$ holds for any
			$t\ge t(r,\varepsilon)$;
		\item If $S$ is in the case \eqref{eq::univariate},
			we have
			$\lim_{r\rightarrow\infty}f_{r,d_\K}^{\psdp}=\lim_{r\rightarrow\infty}f_{r,d_\K}^{\dsdp}=f^*$.
	\end{enumerate}
\end{theorem}
\begin{proof}
	(i)	It is clear that $f_r^*\le f^{\psdp}_{r,t}\le f^{\dsdp}_{r,t}$ holds
	for any $r\ge\lceil d_P/2\rceil$ and $t\ge d_{\K}$. By Theorem
	\ref{th::mainOP} (i), there exists a $r(\varepsilon)\in\N$ such
	that $f^*_r\ge f^*-\varepsilon$ holds for any $r\ge
	r(\varepsilon)$. Thus, (i) follows. 

	(ii) Due to Theorem \ref{th::mainsdp} (i), for any
	$r\ge\lceil d_P/2\rceil$, there exists a
	$t(r,\varepsilon)\in\N$ such that $f^{\psdp}_{r,t}\le
	f^{\dsdp}_{r,t}\le f^*_r+\varepsilon$ holds for any $t\ge
	t(r,\varepsilon)$. Then (ii) follows since $f^*_r\le f^*$ for any
	$r\ge\lceil d_P/2\rceil$ by Theorem \ref{th::mainOP} (i).

	(iii) It is clear by Theorem \ref{th::mainOP} (i) and Theorem
	\ref{th::mainsdp} (iii). 
\end{proof}

\begin{remark}\label{rk::sdp}{\rm
	(i). Theorem \ref{th::approf*} (i) and (ii) implies that we can approximate
	$f^*$ by $f^{\psdp}_{r,t}$ and $f^{\dsdp}_{r,t}$ as closely as
	possible with $r$ and $t$ both large enough. In practice, we can
	let $t=r$ and then
	$\lim_{r\rightarrow\infty}f^{\psdp}_{r,r}=\lim_{r\rightarrow\infty}f^{\dsdp}_{r,r}=f^*$
	under the assumptions in Theorem \ref{th::approf*} (i) and (ii). 

	(ii). Assume that $\tau_\K=1$. By Theorem \ref{th::mainsdp} $($i$)$, for any
	$r\ge\lceil d_P/2\rceil$, there exists $t(r)\in\N$ such that
	$f_{r,t(r)}^{\dsdp}\le f_r^*+1/r$. Denote by 
	$\mL^{(r,t(r))}$ a minimizer
	of $f_{r,t(r)}^{\dsdp}$, then $\{\mL^{(r,t(r))}\}_r$ is a sequence of
	nearly optimal solutions of $(\ref{eq::f*rdual})$ and Theorem
	\ref{th::mainOP} (iii) holds for the corresponding 
	sequence
	$\{\mL^{(r,t(r))}(x)\}_r$.  In particular, when $(\P)$ has
	a unique minizer $u^*$ and $r, t$ are large enough, we can expect
	that the point $\mL^{(r,t)}(x)$ for any
	approximate solution $\mL^{(r,t)}$ of $(\ref{eq::f*rtdual})$ lies
	in a small neighborhood of $u^*$. 
%
}
	$\hfill\square$
\end{remark}

\subsection{S.O.S-Convex case}

Recall Remark \ref{rk::simplification} (iv) and Theorem
\ref{th::mainOP} (iv). We now strengthen Assumption \ref{ass} to 
\begin{assumption}\label{ass2}
	The set $S$ is compact, $-p(x,y)\in\RR[x]$ is s.o.s-convex for any $y\in
	S$ and the Slater condition holds for $\K$.
\end{assumption}

If Assumption \ref{ass2} holds and $f(x)$ is s.o.s-convex, 
then the Lagrangian $L_f(x)$ as defined in (\ref{eq::lag}) is s.o.s 
according to Remark \ref{rk::simplification} (iv).
Like in the general case,
	in the first step, we convert $(\P)$ to 
\begin{equation}\label{eq::sosconvexp}
	\left\{
	\begin{aligned}
		\sup_{\rho,\mu,q}\ \ &\rho&\\
		\text{s.t.}\ &f(x)-\rho=\int_{S} p(x,y)d\mu(y)+q^2,&\\
		&\rho\in\RR,\ \mu\in\mathcal{M}(S),\
		q\in\RR[x]_{\lfloor d_P/2\rfloor}.&
	\end{aligned}\right.
\end{equation}
For $\mathscr{L}\in(\RR[x]_{d_P})^*$, $\rho\in\RR$,
$\mu\in\mathcal{M}(S)$ and
$q\in\RR[x]_{\lfloor d_P/2\rfloor}$,
consider the Lagrange dual function of \eqref{eq::sosconvexp}:
\[
	\begin{aligned}
		L(\rho,\mu,q,\mathscr{L})
		:=&\rho+\mathscr{L}\left(f(x)-\rho-\int_{S}
		p(x,y)d\mu(y)-q^2\right)\\
		=&\mathscr{L}(f)+(1-\mathscr{L}(1))\rho-\int_S
		\mathscr{L}(p(x,y))\ud\mu(y)-\mathscr{L}(q^2).\\
	\end{aligned}
\]
Then, 
\[
	\begin{aligned}
&\sup_{\rho,\mu,q}\
L(\rho,\mu,q,\mathscr{L})\quad\text{s.t.}\ \rho\in\RR,\ \eta\ge 0,\ \mu\in\mathcal{M}(S),\
		q\in\RR[x]_{\lfloor d_P/2\rfloor}\\
		&=\left\{\begin{array}{rll}
			\mathscr{L}(f)&\mbox{if} &\mathscr{L}(1)=1,\\
			&		&\mathscr{L}(q^2)\ge 0,\ \forall\ q\in\RR[x]_{\lfloor d_P/2\rfloor},\\
			&				&\mathscr{L}(p(x,y))\ge 0,\ \forall\ y\in S,\\
			+\infty&\mbox{otherwise}.& 
		\end{array}
		\right.
	\end{aligned}
\]
Hence, the Lagrange dual problem of \eqref{eq::f*r} reads
\begin{equation}\label{eq::sosconvexd}
	\left\{
	\begin{aligned}
		\inf_{\mL\in(\RR[x]_{d_P})^*}\
		&\mathscr{L}(f)&\\
		\text{s.t.}\
		& \mL(1)=1,\  \mL(q^2)\ge 0,\ \forall
		q\in\RR[x]_{\lfloor d_P/2\rfloor},&\\
		&\mathscr{L}(p(x,y))\ge 0,\ \forall y\in S.&\\
	\end{aligned}\right.
\end{equation}

\begin{theorem}\label{th::mainOP2}
	Assume that $f(x)$ is s.o.s-convex and Assumption \ref{ass2} holds, 
	then the following are true.
	\begin{enumerate}[\upshape (i)]
		\item  The optimal values of $(\ref{eq::sosconvexp})$ and
			$(\ref{eq::sosconvexd})$ are both equal to $f^*$ which is
			attainable in $(\ref{eq::sosconvexp})$.
		%
			Moreover, if $f^*$ is attainable in $(\P)$, then so it is
			in $(\ref{eq::sosconvexd})$;
		\item 
			If $\mL^*$ is a minimizer of $(\ref{eq::sosconvexd})$, then 
			$\mL^*(x)=(\mL^*(x_1),\ldots,\mL^*(x_m))$ is a
			minimizer of $(\P)$. 
	\end{enumerate}
\end{theorem}
\begin{proof}
	(i) Denote by $f^*_{\sos}$ the optimal value of
	\eqref{eq::sosconvexp}.  Since Assumption \ref{ass} holds, 
	recalling (\ref{eq::Lag}), there exists $\mu\in\mathcal{M}(S)$
	such that $L_f(x)=f(x)-f^*-\int_S p(x,y)d\mu\ge 0$ for all $x\in\RR^m$.
	Note that the degree of $L_f(x)$
	is even and at most $2\lfloor d_P/2\rfloor$. As $L_f$ is s.o.s,
it holds that  
\[
	f(x)-f^*-\int_S
	p(x,y)d\mu=q^2\quad\text{for some } q\in\RR[x]_{\lfloor d_P/2\rfloor},
\]
which means that \eqref{eq::sosconvexp} is feasible and $f^*_{\sos}\ge f^*$. 
	For any $x\in\K$ and feasible point $(\rho,\mu,q)$ of
	(\ref{eq::sosconvexp}), it holds that $f(x)-\rho\ge 0$ which
	implies that $f^*_{\sos}\le f^*$. Consequently, we have
	$f^*_{\sos}=f^*$. 
	Since (\ref{eq::sosconvexd}) is strictly feasible (see the
	proof of Theorem \ref{th::mainOP} (ii)), (\ref{eq::sosconvexp})
	has an optimal solution and there is no dual gap between (\ref{eq::sosconvexp}) and
	(\ref{eq::sosconvexd}). 

Suppose that $f^*$ is attainable in $(\P)$ at a minimizer $u^*\in\K$. 
Define $\mL^*\in(\RR[x]_{d_P})^*$ by letting
$\mL^*(x^\alpha)=(u^*)^\alpha$ for all $\alpha\in\N^m_{d_P}$, then $f^*$
is attainable in (\ref{eq::sosconvexd}) at $\mL^*$. 

(ii) Compare the feasible set of \eqref{eq::sosconvexd} with the
definition of $\Lambda_{r,t}$ in \eqref{eq::lambda}  and recall the
proof of Theorem \ref{th::main} (ii). Note that to
	show $\Lambda_{r,t}\subseteq\K$ for any $r\ge\lceil d_x/2\rceil$
	and $t\ge d_{\K}$, the constraints
	$\mL(\Theta_{k})\le 1$ in (\ref{eq::lambda})
	are redundant.  
	Moreover, the inequality \eqref{eq::contradiction2} still holds
	for $\mL$ feasible to \eqref{eq::sosconvexd}.
Hence, we can obtain that $\mL^*(x)\in\K$. 
As $f(x)$ is s.o.s-convex, by \cite[Theorem 2.6]{LasserreConvex}, the
extension of Jensen's inequality
$f(\mL^*(x))\le\mL^*(f)=f^*$ holds, which implies that
$\mL^*(x)$ is a minimizer of $(\P)$.
\end{proof}

\vskip 10pt 
In the same way as we derive the SDP relaxations \eqref{eq::f*rt} and
\eqref{eq::f*rtdual} from \eqref{eq::f*r} and \eqref{eq::f*rdual}, 
we next obtain corresponding SDP relaxations of (\ref{eq::sosconvexp}) and
(\ref{eq::sosconvexd}) as
\begin{equation}\label{eq::sosconvext}
	\left\{
	\begin{aligned}
		f_{t}^{\psdp}:=\sup_{\rho,\mH,q}\ &\rho&\\
		\text{s.t.}\
		&f(x)-\rho=\mH(p(x,y))+q^2,&\\
		&\rho\in\RR,\ \mH\in(\qm_t(G))^*,\ q\in\RR[x]_{\lfloor d_P/2\rfloor}.&
	\end{aligned}\right.
\end{equation}
and its dual 
\begin{equation}\label{eq::sosconvextdual}
	\left\{
	\begin{aligned}
		f_{t}^{\dsdp}:=\inf_{\mL}\
		&\mathscr{L}(f)&\\
		\text{s.t.}\
		&\mL\in(\RR[x]_{2r})^*,\ \mL(q^2)\ge 0,\ \forall
		q\in\RR[x]_{\lfloor d_P/2\rfloor},&\\
		&\mL(1)=1,\ \mL(p(x,y))\in\qm_t(G).&\\
	\end{aligned}\right.
\end{equation}

\begin{theorem}\label{th::mainsdp2}
	Assume that $f(x)$ is s.o.s-convex and Assumption \ref{ass2}
	holds, 
	then the following are true.
	\begin{enumerate}[\upshape (i)]
		\item If $\qm(G)$ is Archimedean and the Slater
			condition holds for $\K$, then $\lim_{t\rightarrow\infty}f_{t}^{\psdp}
			=\lim_{t\rightarrow\infty}f_{t}^{\dsdp}=f^*$;
		\item For some order $t\ge d_{\K}$, if the flat extension
			condition 
			holds for $\mH^*$ in the solution $(\rho^*,\mH^*,q^*)$
			of $(\ref{eq::sosconvext})$, then $f_{t}^{\psdp}=f^*$;
		\item 
		Let $\{\mL^{(t)}\}_t$ be a sequence of nearly
	optimal solutions of $(\ref{eq::sosconvextdual})$ and
	$\mL^{(t)}(x)=(\mL^{(t)}(x_1),\ldots,\mL^{(t)}(x_m))$. 
	For any convergent subsequence
	$\{\mL^{(t_i)}(x)\}_i$ of $\{\mL^{(t)}(x)\}_t$, 
		$\lim_{i\rightarrow\infty}\mL^{(t_i)}(x)$ is a minimizer of
		$(\P)$. Consequently,
		if $\{\mL^{(t)}(x)\}_t$ is bounded and $u^*$ is a unique minimizer
		of $(\P)$, then $\lim_{t\rightarrow\infty}{\mL^{(t)}(x)}=u^*$.
	\item If $S$ is in the case \eqref{eq::univariate},
		then we have 
$f^{\psdp}_{d_{\K}}=f^{\dsdp}_{d_{\K}}=f^*$.
		If $(\P)$ is solvable, then $u^*$ is a minimizer of
		$(\P)$ if and only if there exists a minimizer
		$\mL^*$ of $(\ref{eq::sosconvextdual})$
		with $t=d_{\K}$ such that $\mL^*(x)=u^*$. 
	\end{enumerate}
\end{theorem}
\begin{proof}
	(i) and (ii): See Theorem \ref{th::mainOP2} (i) and the proofs of
	Theorem \ref{th::mainsdp} (i) and (ii).
	
	(iii): Since $f(x)$ is s.o.s-convex, due to the
	extended Jensen's inequality \cite[Theorem 2.6]{LasserreConvex},
	it holds that $f(\mL^{(t)}(x))\le \mathscr{L}^{(t)}(f)$ and
	therefore
	$f(\lim_{t\rightarrow\infty}\mL^{(t)}(x))\le\lim_{t\rightarrow\infty}\mL^{(t)}(f)=f^*$.
	From the proofs of Theorem \ref{th::main} (ii) and Theorem
	\ref{th::mainOP2} (ii), it is easy to see that 
	the sequence
	$\{\mL^{(t)}(x)\}\subset\K$ and hence
	$\lim_{t\rightarrow\infty}\mL^{(t)}(x)\in\K$. 
    Thus, $\lim_{i\rightarrow\infty}\mL^{(t_i)}(x)$ is a minimizer of
	$(\P)$.

	(iv): By Theorem \ref{th::mainOP2} (i) and the weak duality, it
	holds that $f^*\le f^{\psdp}_t\le f^{\dsdp}_t$ for any $t\ge
	d_{\K}$.
	For any $\varepsilon>0$, there exists a point
	$u^{(\varepsilon)}\in\K$ such that
	$f(u^{(\varepsilon)})\le f^*+\varepsilon$. 
	Define $\mL^\varepsilon\in(\RR[x]_{d_P})^*$ by letting
	$\mL^\varepsilon(x^\alpha)=(u^{(\varepsilon)})^\alpha$ for all $\alpha\in\N^m_{d_P}$.
By the representations of univariate polynomials nonnegative on an interval (c.f.
\cite{PR2000,Laurent_sumsof}),
$\mL^{(\varepsilon)}$ is feasible to (\ref{eq::sosconvextdual}) with
$t=d_{\K}$, which
implies that $f^{\dsdp}_{d_{\K}}\le f^*+\varepsilon$. Since
	$\varepsilon$ is abitrary, it holds that
	$f^{\psdp}_{d_{\K}}=f^{\dsdp}_{d_{\K}}=f^*$.

	Clearly, we only need to prove the ``if'' part. Since Assumption
	\ref{ass2} holds,
from the proofs of Theorem \ref{th::main} (ii) and Theorem
	\ref{th::mainOP2} (ii), 
	we have $\mL^*(x)\in\K$.
	As $f(x)$ is s.o.s-convex, due to the extended Jensen's inequality
	\cite[Theorem 2.6]{LasserreConvex},
	it holds that $f^*\le f(\mL^*(x))\le \mathscr{L}^*(f)=f^*$. 
	Thus, $\mL^*(x)$ is a minimizer of $(\P)$. 
\end{proof}

\begin{remark}{\rm
			(i) Note that we do not require $\K$ to be compact in Theorem
				\ref{th::mainOP2} and \ref{th::mainsdp2}. 

			(ii) In the special case when $f(x),\ p(x,y)$ are linear
				in $x$ for every $y\in S$, the SDP relaxation
				\eqref{eq::sosconvext} agrees with the SDP
				relaxation of generalized problems of moments
				proposed in \cite{Lasserre2008}. 
}
	$\hfill\square$
\end{remark}

\begin{example}\label{ex::OP} Now we consider two convex semi-infinite polynomial
	programming problems using the sets $\K$ defined in Example
	\ref{ex::sdr} (\ref{ex::K1}) and (\ref{ex::K2}).
	Notice that the
	constraints in the dual SDP relaxations $(\ref{eq::f*rtdual})$ and 
	$(\ref{eq::sosconvextdual})$
	can be easily generated by {\sf Yalmip}. Hence, we solve the following problems
	using these corresponding dual SDP relaxations, which can also give
	us some informations on the minimizers of the problems. 
	\begin{enumerate}[(1).]
		\item Recall the sets $\K$ and $S$ defined in Example
			\ref{ex::sdr} (\ref{ex::K1})
			where the polynomial $p(x_1,x_2,y_1,y_2)\in\RR[x_1,x_2]$
			is convex but not s.o.s-convex for every $y\in S$. 
			\cite{APgap} constructed a polynomial
			\[
				\begin{aligned}
				\td{f}(x_1,x_2)=&89-363x_1^4x_2+\frac{51531}{64}x_2^6-\frac{9005}{4}x_2^5+
					\frac{49171}{16}x_2^4+721x_1^2-2060x_2^3-14x_1^3+\frac{3817}{4}x_2^2\\
					&+363x_1^4-9x_1^5+77x_1^6+316x_1x_2+49x_1x_2^3-2550x_1^2x_2-968x_1x_2^2
					+1710x_1x_2^4\\
					&+794x_1^3x_2+\frac{7269}{2}x_1^2x_2^2-\frac{301}{2}x_1^5x_2
					+\frac{2143}{4}x_1^4x_2^2+\frac{1671}{2}x_1^3x_2^3+\frac{14901}{16}x_1^2x_2^4\\
					&-\frac{1399}{2}x_1x_2^5-\frac{3825}{2}x_1^3x_2^2-\frac{4041}{2}x_1^2x_2^3-364x_2+48x_1.
				\end{aligned}
			\]
			$($see {\upshape \cite[(5.2)]{APgap}}$)$ which is
			convex but not s.o.s-convex. In order to illustrate the
			efficiency of the SDP relaxations $(\ref{eq::f*rtdual})$ better, 
			we shift and scale  $\td{f}$
			to define $f(x_1,x_2):=\td{f}(x_1-1,x_2-1)/10000$, which
			is still convex but not s.o.s-convex. Then, consider the
			problem $\min_{x\in\K} f(x_1,x_2)$, where $d_P=d_x=8$.
			Letting $r=t=4$, we get $f_{4,4}^{\dsdp}=0.15234$
			achieved at $\mL^{(4,4)}$ and an approximate minimizer 
			$\mL^{(4,4)}(x)=(0.4245,0.6373)$. To show the accuracy of
			the solution, we draw some contoure lines of $f$, including
			$f(x_1,x_2)=0.15234$, and mark the point
			$\mL^{(4,4)}(x)$ by red `$+$' in Figure \ref{fig::fmin1} (left). As we
			can see, the line $f(x_1,x_2)=0.15234$ is almost tangent
			to $\K$ at the point $\mL^{(4,4)}(x)$. 
		\item Recall the sets $\K$ and $S$ defined in Example
			\ref{ex::sdr} (\ref{ex::K2}). Let $f(x_1,x_2):=(x_1-1)^2+x_2^2$, i.e.,
			the square of the distance function of a point to $(1,0)$,
			and consider the problem $\min_{x\in\K} f(x_1,x_2)$. Then,
			the polynomials $f(x_1,x_2)$ and $-p(x_1,x_2,y_1,y_2)$ for
			all $y\in S$ are s.o.s-convex. As $d_{\K}=1$, solving the SDP
			relaxation $(\ref{eq::sosconvextdual})$ with $t=1$, we get
			$f_1^{\dsdp}=0.80942$ achieved at $\mL^{(1)}$ and an
			approximate minimizer 
			$\mL^{(1)}(x)=(0.1311,-0.2335)$. The corresponding
			contoures and the point $\mL^{(1)}(x)$ are shown in Figure
			\ref{fig::fmin1} (right). 

	\end{enumerate}
\begin{figure}
	\centering
	\caption{\label{fig::fmin1} The sets $\K$ and contoure lines of
	$f$ in Example \ref{ex::OP}.}
\scalebox{0.5}{
	\includegraphics[trim=150 220 150 220,clip]{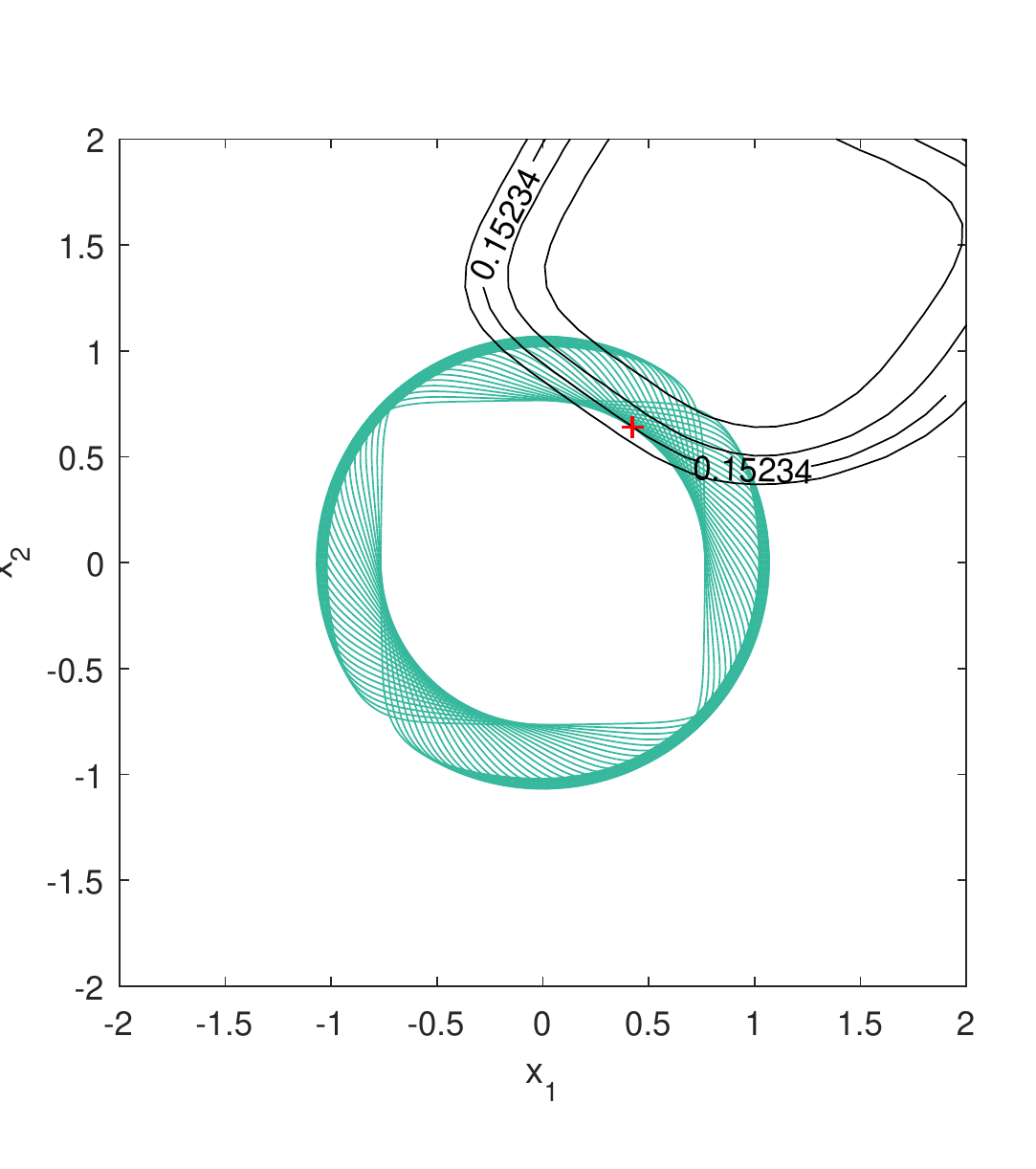}\qquad
	\includegraphics[trim=150 220 150 220,clip]{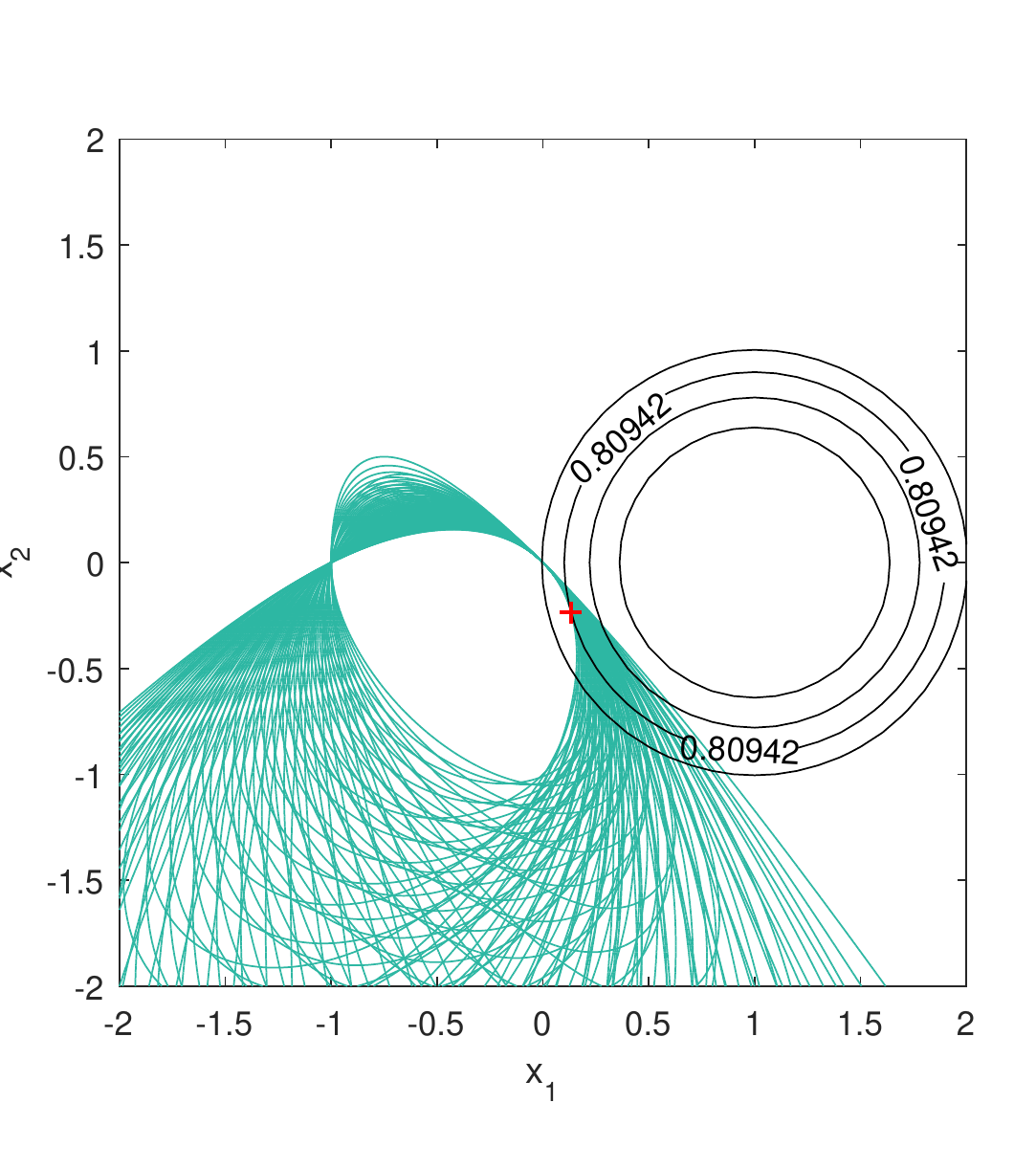}
}
\end{figure}
\end{example}

To end this section, we compare our SDP relaxation method for the
convex semi-infinite polynomial programming problem $(\P)$ with the approach
given in \cite{SIPPSDP}, which can also solve $(\P)$ via SDP
relaxations. In fact, the method proposed in \cite{SIPPSDP} is based
on the exchange scheme and works for semi-infinite polynomial
programming problems without requiring convexity. 

Generally speaking, given a
finite subset $S_k\subseteq S$
in an iteration, one obtains at least one global minimizer $u^{(k)}$ of
$f(x)$ under the associated finitely many constraints and then
compute the global minimum $p^k$ and minimizers
$y^{(1)},\ldots,y^{(l)}$ of
$p(u^{(k)},y)$ over $S$. If $p^k\ge 0$, stop; otherwise, update
$S_{k+1}=S_k\cup\{y^{(1)},\ldots,y^{(l)}\}$ and proceed to the next
iteration. Therefore, to guarantee the
success of the exchange method, the subproblems in each iteration
need to be globally solved and at least one minimizer of each
subproblem can be extracted.  The subproblems can be solved by
Lasserre's SDP relaxation method and minimizers can be
extracted when the flat extension condition holds. 

However, Lasserre's SDP
relaxation method for the lower level subproblem of  minimizing
$p(u^{(k)},y)$ over $S$ does not necessarily have finite
convergence. Even it does, a minimizer for the lower level subproblem
could not be extracted. In particular, when there are infinitely many
minimizers, the flat extension condition fails (c.f.
\cite[Sec. 6.6]{Laurent_sumsof}). 
For polynomial optimization problems with
generic coefficients data, 
according to \cite[Theorem 1.2]{NieFiniteCon} and 
\cite[Proposition 2.1]{NR2009}, 
there are finitely many minimizers and Lasserre's SDP relaxation
method has finite convergence. 
However, even the coefficients data in
$(\P)$ is generic, we are not clear about the success of the method
in \cite{SIPPSDP} applied to $(\P)$. 
It is because in the lower level subproblems, the coefficients of
$p(u^{(k)},y)$ depend on the solutions of the upper level subproblem
of the same interation.

For example, consider the problem
\[
	\left\{
	\begin{aligned}
		\min_{x\in\RR^2}&\ x_1^2+x_2^2\\
		\text{s.t.}&\ (x_1+x_2-1)(y_1^2+y_2^2)+(x_2-1)(y_3^2-1)\ge
	0,\\
	&\ \forall\ y\in\{y\in\RR^3\mid 1-y_1^2-y_2^2\ge 0,\ 1-y_3^2\ge
	0\}. 
	\end{aligned}\right.
\]
It is easy to see that the feasible set is $\{x\in\RR^2\mid
x_1+x_2-1\ge 0,\ 1-x_2\ge 0\}$ and the minimizer is
$(\frac{1}{2},\frac{1}{2})$. If we choose the intial set
$S_0=\{(0,0,0)\}$, then upper level subproblem has a unique minimizer
$u^{(0)}=(0,0)$. Then for the lower level subproblem of minimizing
$p(u^{(0)},y)=-y_1^2-y_2^2-y_3^2+1$ over $S$, it clear that the
solution set is $\{y\in\RR^3\mid y_1^2+y_2^2=1, y_3^2=1\}$ which is
infinite and the flat extension condition does not apply for Lasserre's SDP
relaxations. As none of the minimizers can be extracted, the method in
\cite{SIPPSDP} fails for this problem. Since the objective functions is
s.o.s-convex and Assumption \ref{ass2} holds, we can solve the above
problem by our SDP relaxations $(\ref{eq::sosconvextdual})$.  Let
$t=1$, we get $f_1^{\dsdp}=0.2500$ achieved at $\mL^{(1)}$ and an
approximate minimizer $\mL^{(1)}(x)=(0.5000,0.5000)$.

\section*{Acknowledgments}
The authors are very grateful for the comments of two anonymous
referees which helped to improve the presentation.
The first author was supported by
the Chinese National Natural Science Foundation under grants 11401074, 11571350.
The second author was supported by the Chinese National Natural
Science Foundation under grant 11801064, the Foundation of Liaoning Education
Committee under grant LN2017QN043.


\end{document}